\documentclass[12pt,reqno]{amsart}
\usepackage{amsmath,amssymb,amsfonts,amsthm}
\usepackage[right=2.5cm, left=2.5cm, top=2.5cm, bottom=2.5cm]{geometry}
\usepackage{tikz}
\usetikzlibrary{positioning, arrows.meta, shapes.geometric}
\usetikzlibrary{matrix}
\usepackage{graphicx}
\usepackage{pgfplots}
\usepackage{subcaption} 
\usepgfplotslibrary{fillbetween}
\usetikzlibrary{patterns}
\DeclareMathAlphabet{\mathpzc}{OT1}{pzc}{m}{it}

\newtheorem{theorem}{Theorem}[section]
\newtheorem{lema}[theorem]{Lemma}

\newtheorem{co}[theorem]{Corollary}
\theoremstyle{definition}
\newtheorem{de}[theorem]{Definition}

\theoremstyle{remark}
\newtheorem{remark}[theorem]{Remark}

\numberwithin{equation}{section}

\begin{document}
	
\title[Critical thresholds for super-diffusive memory equations]{Critical thresholds and instantaneous norm inflation for super-diffusive integro-differential equations}

\author[B. de Andrade]{Bruno de Andrade}
\address[B. de Andrade]{Departamento de Matem\'atica\\
	Universidade Federal de Sergipe\\
	S\~ao Crist\'ov\~ao - SE\\
	Brazil.}
\email{bruno@mat.ufs.br}

\begin{abstract}
	This manuscript investigates the Cauchy problem for a class of nonlinear integro-differential equations governing anomalous super-diffusive transport in $\mathbb{R}^N$. The linear dynamics are driven by a dual-scale memory kernel whose Laplace transform is sectorial and exhibits distinct power-law asymptotics at high and low frequencies. This super-diffusive structure precludes the infinite regularizing capacity characteristic of classical parabolic theory; consequently, the associated resolvent operator possesses a heavy algebraic tail in Fourier space, acting as a pseudo-differential operator in the H\"ormander class $S^{-2}_{1,0}$ and restricting spatial smoothing. By establishing rigorous $L^q-L^p$ multiplier estimates, the critical Lebesgue threshold $q_c$ for local well-posedness is determined. To demonstrate the sharpness of this threshold, instantaneous norm inflation---and consequent ill-posedness---is proven in the supercritical regime $1 < q < q_c$. Furthermore, tracking the structural crossover to the long-time relaxation parameter resolves the global asymptotic dynamics. The nonlocal Fujita-type critical exponent $\rho_F$ is identified, and global-in-time existence along with algebraic decay is established for small initial data in intersection spaces, provided the nonlinearity remains supercritical and overcomes the structural algebraic barrier connecting the dual scales. This general framework applies directly to canonical physical models, including Cole-Cole fractional retardation and multi-scale Prabhakar memory.
\end{abstract}


\maketitle

\section{Introduction}
Partial differential equations driven by non-local memory operators arise in diverse mathematical and applied contexts. The study of anomalous diffusion, in particular, encompasses transport processes where the mean squared displacement scales non-linearly in time as $t^\gamma$ with $\gamma \neq 1$, a phenomenon that naturally emerges in complex heterogeneous media, ranging from turbulent plasma and subsurface tracer transport to cellular biology and non-Newtonian viscoelastic fluids \cite{Metzler2000}. At the mesoscopic level, these dynamics are frequently modeled via Continuous Time Random Walks (CTRW), where the decoupling of heavy-tailed waiting times and jump length distributions generates non-Markovian macroscopic limits \cite{Meerschaert2012}. When the temporal waiting time distribution exhibits power-law decay, the classical instantaneous derivative is replaced by a non-local memory operator, yielding integro-differential equations that interpolate between standard heat and wave dynamics.

This manuscript investigates the Cauchy problem for a class of nonlinear integro-differential equations with memory, governing super-diffusive anomalous transport in $N$-dimensional space $\mathbb{R}^N$ ($N \ge 1$), defined by
\begin{equation}\label{eq:main}
	\begin{cases}
		\partial_t u(x,t) = \displaystyle \int_0^t g(t-s)\Delta u(x,s)ds + f(u(x,t)), & \text{in } \mathbb{R}^N \times (0,\infty), \\
		u(x,0) = u_0(x), & \text{in } \mathbb{R}^N.
	\end{cases}
\end{equation}
For initial data $u_0 \in L^q(\mathbb{R}^N)$ with $1 \le q < \infty$, the nonlinear source is assumed to exhibit power-type growth $f(u) = u|u|^{\rho-1}$ with $\rho > 1$. The memory kernel $g \in L^1_{\text{loc}}(\mathbb{R}^+)$ dictates the underlying linear dynamics. Rather than restricting the analysis to a single scale-invariant fractional operator, the present framework addresses a broad class of admissible kernels  (formalized in Section 2) whose Laplace transforms are sectorial and characterized by distinct power-law asymptotics at high and low frequencies, governed by exponents $\alpha_\infty \in (0,1)$ and $\alpha_0 \in [0,1)$, respectively. Since the memory kernel acts as a temporal convolution directly over the Laplacian, the effective differential order of the system at high frequencies shifts to $1+\alpha_\infty$, implying that the short-time mean squared displacement scales proportionally to $t^{1+\alpha_\infty}$. This scaling exponent strictly exceeds one, ensuring the model captures a super-diffusive regime.

\subsection{Physical motivation and the multi-scale memory kernel}
The adoption of distinct temporal limits for the kernel $g$ arises from the concept of structural crossovers in complex media. The standard time-fractional equation corresponds to the Riemann-Liouville kernel $g(t) = \frac{t^{\alpha-1}}{\Gamma(\alpha)}$. As established by Baeumer and Meerschaert \cite{Baeumer2001}, such time-fractional Cauchy problems emerge when long-range temporal correlations lead to the subordination of classical uncorrelated motion. This scale invariance collapses the high and low-frequency regimes into a single anomalous parameter $\alpha_\infty = \alpha_0 = \alpha$. However, heterogeneous media often feature transport mechanisms that transition between distinct anomalous regimes, or relax back to classical diffusion over extended time scales, behavior which cannot be captured by a single fractional index.

The generalized framework developed herein captures these complex dynamics. For instance, in viscoelastic materials exhibiting distinct retardation times, the memory kernel is typically modeled as a linear combination of fractional derivatives (multi-term fractional diffusion). In this scenario, short-time particle interactions are dominated by the most singular exponent $\beta$, whereas long-time transport is governed by the slower relaxation exponent $\alpha$, thereby establishing a dual-scale regime where $\alpha_\infty = \beta < \alpha = \alpha_0$. 

Another canonical application involves fractional retardation processes, which arise when anomalous transport undergoes macroscopically governed transitions toward standard Gaussian profiles over long horizons. This behavior is typified by the Cole-Cole relaxation model. Originally proposed by Cole and Cole \cite{ColeCole1941} to describe broad dispersion and absorption in dielectrics---where the complex dielectric locus traces a depressed circular arc in the complex plane---this empirical formulation has been extensively adapted to anomalous transport \cite{Mainardi2010}. The corresponding memory transform takes the algebraic form $\hat{g}(\lambda) = (\lambda^\alpha + \gamma)^{-1}$ for $\gamma > 0$. At short times ($|\lambda| \to \infty$), the system exhibits highly singular super-diffusive behavior characterized by $\alpha_\infty = \alpha \in (0,1)$. Conversely, at long times ($|\lambda| \to 0$), the memory dynamics undergo a structural crossover to a classical random walk limit, yielding $\alpha_0 = 0$. Modeling this spectral dichotomy without introducing unmanageable singular poles on the principal branch cut fundamentally distinguishes this mathematical architecture from standard fractional parabolic analysis.

\subsection{State of the art and mathematical challenges}
The interplay among diffusion, historical memory, and nonlinear amplification remains a central focus in nonlinear analysis. For the classical heat equation (where $g$ acts as the Dirac delta distribution), the seminal work of Fujita \cite{Fujita1966} established the critical exponent $\rho_F = 1 + 2/N$, separating global existence from finite-time blow-up. Weissler \cite{Weissler1981} subsequently developed the optimal Lebesgue space well-posedness theory for the local parabolic problem, relying on the infinite regularizing capacity of the Gaussian heat kernel. Building on this functional framework, Giga \cite{Giga1986} established abstract $L^p$ estimates and semigroup methods to prove the existence and regularity of solutions for semilinear parabolic equations. Similarly, for the semilinear wave equation, the search for critical exponents and finite-time blow-up of solutions was notably advanced by Kato \cite{Kato80}, who identified the structural barriers separating global solutions from singularity formation in nonlinear hyperbolic equations. 

When transitioning to systems with historical memory, the early foundational works of Gurtin and Pipkin \cite{Gurtin1968}, Nunziato \cite{Nunziato71}, and Miller \cite{Miller1978} on integrodifferential equations for rigid heat conductors set the stage for studying generalized linear transport dynamics. Subsequent models of fractional diffusion and wave equations were rigorously formalized by Schneider and Wyss \cite{Schneider1989}, expressing fundamental solutions in terms of Fox functions. Concurrently, the study of anomalous transport driven by spatial non-locality has gained prominence. Guedda and Kirane \cite{Guedda2001} investigated equations with fractional spatial operators, extending the Fujita phenomenon to the fractional Laplacian framework. In the context of scaling-invariant spaces, Ferreira and Villamizar-Roa \cite{FerreiraVillamizar2006} proved the uniqueness and long-time asymptotic stability of self-similar solutions. Cazenave et al. \cite{Cazenave2006} further expanded the theory to parabolic equations where a classical time derivative is coupled with a time-fractional integral acting exclusively on the nonlinear reaction (see also Fino and Kirane \cite{Fino2012}). Expanding the scope to fully nonlocal diffusion equations, Kemppainen, Siljander, and Zacher \cite{Kemppainen2017} advanced the decay theory for mild and weak solutions. Utilizing Fourier analysis and energy methods, they established quantitative $L^p$ and optimal $L^2$ decay rates, revealing a critical dimension phenomenon where the decay ceases to improve in high dimensions (see also Pr\"uss \cite{Pruss1993} for foundational resolvent operator theories, and Vergara and Zacher \cite{VergaraZacher2015} for optimal decay estimates via energy methods). Focusing on purely time-fractional diffusion featuring instantaneous nonlinear reactions, de Andrade, Siracusa, and Viana \cite{deAndradeSiracusaViana} recently established local well-posedness and a blow-up alternative, demonstrating that the Fujita-type critical exponent depends on the derivation fractional order.

In the related context of higher-order non-local dynamics, the literature has evolved from bounded geometries to whole-space formulations. For the fractional diffusion-wave equation where the analytic order satisfies $\alpha \in (1,2)$, establishing mild solvability, smoothing estimates, and explicit finite-time blow-up thresholds in bounded domains $\Omega \subset \mathbb{R}^N$ was recently achieved in critical and subcritical Lebesgue spaces \cite{CCdeA26, deA26, deAndradeSantos2026}. Parallel to these boundary-value developments, the analysis in the entire space $\mathbb{R}^N$ was advanced by D'Abbicco, Ebert, and Picon \cite{DabbiccoEbertPicon2019}. In their framework, the time-fractional Caputo derivative acts comprehensively over the system, meaning that the non-local memory operator encompasses both the diffusive and the reactive terms. Whether through the geometric confinement of bounded domains or the comprehensive history convolution on the non-linear source, these classical settings inherently benefit from explicit Mittag-Leffler representations or additional fractional integration smoothing effects, which naturally damp high-frequency oscillatory singularities.

However, the super-diffusive regime ($\alpha_\infty > 0$) induced by the dual-scale kernel $g$ in the present model fundamentally alters this analytical landscape. By coupling a strictly instantaneous local-in-time reaction with a generalized anomalous dissipation in the entire space $\mathbb{R}^N$, the mild flow lacks any non-local history damping on the source, thereby preventing the direct application of explicit Mittag-Leffler functions. Instead, the transfer function $\widehat{S}(t,\xi)$ exhibits a heavy algebraic tail at high frequencies; consequently, the fundamental solution lacks global integrability, precluding infinite regularizing capacity. The underlying resolvent operator acts intrinsically as a pseudo-differential operator belonging to the H\"ormander class $S^{-2}_{1,0}$, restricting spatial smoothing to a gain of exactly two derivatives and enforcing a strict topological barrier that requires the dimensional constraint $\frac{1}{q} - \frac{1}{p} < \frac{2}{N}$.

Overcoming this barrier to close the fixed-point contraction mapping constitutes the core analytical challenge of this manuscript, particularly within a dual-scale kernel framework where the singular behavior shifts structurally between the short-time exponent $\alpha_\infty$ and the long-time relaxation parameter $\alpha_0$. Resolving this scale mismatch requires the development of  Fourier multiplier estimates capable of tracking these cross-frequency transitions without degeneration. Furthermore, to demonstrate the geometric sharpness of this critical threshold, a non-local norm inflation result based on asymptotic frequency modulations is implemented, proving that the supercritical data-to-solution map undergoes an immediate collapse near the origin. This unified approach provides a rigorous foundation for both local flow stability and the structural ill-posedness of super-diffusive transport media.

\subsection{Main contributions}
This manuscript aims to establish the optimal local and global well-posedness theory for \eqref{eq:main}, alongside a rigorous proof of ill-posedness in the supercritical Lebesgue regime. Deploying a combination of complex contour deformation, Fourier multiplier theory, and fractional Sobolev embeddings yields  $L^q-L^p$ bounds for the resolvent operator family $(S(t))_{t\ge0}$.

The high-frequency and low-frequency smoothing indices are defined, respectively, by
\begin{equation*}
	\gamma_{q,p} = \frac{N(1+\alpha_\infty)}{2}\left(\frac{1}{q} - \frac{1}{p}\right), \quad \text{and} \quad \beta_{m,p} = \frac{N(1+\alpha_0)}{2}\left(\frac{1}{m} - \frac{1}{p}\right).
\end{equation*}
As demonstrated in Section 4, the short-time dynamics are governed strictly by the high-frequency exponent $\alpha_\infty$. The critical Lebesgue threshold for local well-posedness is therefore identified as
\begin{equation*}
	q_c := \frac{N(\rho-1)(1+\alpha_\infty)}{2}.
\end{equation*}
Theorem \ref{thm:local_existence} secures local Hadamard well-posedness in $L^q(\mathbb{R}^N)$ for $q \ge q_c$, utilizing carefully calibrated time-weighted Banach spaces to absorb the severe temporal singularity $\gamma_{q,p}$ near the origin. 

Confirming the sharpness of this threshold, Theorem \ref{thm:ill_posedness} establishes the supercritical ill-posedness of the problem. For any $1 < q < q_c$, localized initial data are constructed to display instantaneous norm inflation. Adapting the asymptotic techniques of Christ, Colliander, and Tao \cite{Christ2003} to the non-local operator behavior proves that the data-to-solution map is discontinuous at the origin in the $L^q$ topology.

Section 5 shifts the focus to the global asymptotic dynamics. Theorem \ref{thm:global_existence} resolves the dual-scale complexity of the kernel. For solutions to exist globally and avoid finite-time blow-up, the nonlinear source must balance against the long-time relaxation parameter $\alpha_0$. The nonlocal critical mass emerges as $m_c = \frac{N(\rho-1)(1+\alpha_0)}{2}$, generating the corresponding Fujita-type critical exponent $\rho_F = 1 + \frac{2}{N(1+\alpha_0)}$. The analysis demonstrates that if the nonlinear reaction overcomes the novel spectral bridge constraint 
\[\rho > \rho_0 := \max\left( \rho_F, \frac{1+\alpha_\infty}{1+\alpha_0} \right),\] 
a unique global solution exists provided the initial data is sufficiently small in the intersection space $L^{q_c}(\mathbb{R}^N) \cap L^m(\mathbb{R}^N)$, where $m \in \left(\frac{q_c}{\rho}, m_c\right)$. Finally, Corollary \ref{cor:asymptotic_decay} confirms that despite the non-local non-linear memory, the global solution inherits the unperturbed linear asymptotic decay profile $\mathcal{O}(t^{-\beta_{m,p}})$.

\noindent \textbf{A concrete application.} To clarify the physical readability and immediate applicability of these abstract thresholds, consider the anomalous propagation of a structural disturbance in a three-dimensional heterogeneous viscoelastic medium governed by a canonical Cole-Cole memory kernel $g(t) = t^{-2/3} E_{1/3, 1/3}(-2 t^{1/3})$ and a cubic nonlinear amplification source:
\begin{equation*}
	\begin{cases}
		\partial_t u(x,t) = \displaystyle \int_0^t g(t-s) \Delta u(x,s)ds + u(x,t)^3, & \text{in } \mathbb{R}^3 \times (0,\infty), \\
		u(x,0) = u_0(x), & \text{in } \mathbb{R}^3.
	\end{cases}
\end{equation*}
Here, the high-frequency anomalous exponent is $\alpha_\infty = 1/3$, while the long-time relaxation parameter corresponds to the classical diffusion crossover $\alpha_0 = 0$. Since the nonlinearity is cubic, the critical spatial regularity index verifies 
\begin{equation*}
	q_c = \frac{3(3-1)(1+1/3)}{2} = 4,
\end{equation*}
and, consequently, $L^{q_c}(\mathbb{R}^3)$ simplifies to the standard space $L^{4}(\mathbb{R}^3)$. The non-local Fujita-type critical exponent for this geometry reduces to $\rho_F = 1 + \frac{2}{3(1+0)} = \frac{5}{3}$. Since the physical reaction power satisfies $\rho = 3 > \max\left(\frac{5}{3}, \frac{4}{3}\right) = \frac{5}{3}$, the established framework guarantees that any small initial disturbance belonging to the intersection of classical energy spaces, namely $u_0 \in L^4(\mathbb{R}^3) \cap L^2(\mathbb{R}^3)$, yields a unique global-in-time mild solution that decays algebraically toward equilibrium, avoiding the finite-time singularity formation inherent to supercritical configurations.

\subsection{Final remarks and generalizations}\label{subsec:final_remarks}
In conclusion, the algebraic power-type nonlinearity $f(u) = u|u|^{\rho-1}$ selected in \eqref{eq:main} serves as the canonical framework required to isolate and establish the exact, sharp geometric thresholds of dual-scale transport. While the core topological mechanism---namely, the duality between the high-frequency H\"ormander class $S^{-2}_{1,0}$ and the low-frequency relaxation parameter $\alpha_0$---is expected to underpin a broader class of anomalous evolution problems, any explicit generalization to non-homogeneous or gradient-dependent sources would fundamentally alter the structure of these critical indices.

Specifically, treating a locally Lipschitz continuous source $f(u)$ with H\"older-type growth near the origin, or an asymptotic power law behavior as $|u| \to \infty$, would preserve the qualitative nature of the local flow and the instantaneous norm inflation. However, if the reaction depends explicitly on space-time potentials $\phi(x,t)$ or the spatial gradient $\nabla u$, the implementation of fixed-point contractive arguments would necessitate a complete restructuring of the functional framework. Absorbing gradient-dependent terms without violating the integrability barrier would require the development of a time-weighted Sobolev architecture $W^{1,q}(\mathbb{R}^N)$ and would induce complex dimensional shifts in the thresholds. Tracking the arithmetic variations of these shifted critical exponents demands disparate technical treatments, placing such extensions outside the scope of the present manuscript. Therefore, the analysis of the canonical model presented herein provides a foundational mathematical baseline upon which subsequent generalized multi-scale frameworks can be anchored.

\noindent \textbf{Organization of the paper.} The remainder of this work is organized as follows. Section 2 introduces the mathematical setting, defining the generalized class of admissible dual-scale kernels and establishing the mild integral framework. Section 3 is devoted to the core linear theory, extracting scale-invariant $L^q-L^p$ smoothing bounds via H\"ormander multiplier theory and dyadic frequency partitions. In Section 4, the critical Lebesgue threshold $q_c$ is determined to prove local Hadamard well-posedness, followed by a rigorous demonstration of instantaneous norm inflation in the supercritical regime. Section 5 addresses the global dynamics; it formalizes the spectral bridge constraint, establishes global-in-time existence, and proves the asymptotic linear decay of solutions. Finally, Appendix A provides the uniform spatial localization bounds via Peetre's inequality required for the ill-posedness analysis.

	%%%%%%%%%%%%%%%%%%%%%%%%%%%%%%%%%%%%%	%%%%%%%%%%%%%%%%%%%%%%%%%%%%%%%%%%%%%	%%%%%%%%%%%%%%%%%%%%%%%%%%%%%%%%%%%%%	%%%%%%%%%%%%%%%%%%%%%%%%%%%%%%%%%%%%%	%%%%%%%%%%%%%%%%%%%%%%%%%%%%%%%%%%%%%	%%%%%%%%%%%%%%%%%%%%%%%%%%%%%%%%%%%%%	%%%%%%%%%%%%%%%%%%%%%%%%%%%%%%%%%%%%%	%%%%%%%%%%%%%%%%%%%%%%%%%%%%%%%%%%%%%	%%%%%%%%%%%%%%%%%%%%%%%%%%%%%%%%%%%%%	%%%%%%%%%%%%%%%%%%%%%%%%%%%%%%%%%%%%%	%%%%%%%%%%%%%%%%%%%%%%%%%%%%%%%%%%%%%	%%%%%%%%%%%%%%%%%%%%%%%%%%%%%%%%%%%%%	%%%%%%%%%%%%%%%%%%%%%%%%%%%%%%%%%%%%%	%%%%%%%%%%%%%%%%%%%%%%%%%%%%%%%%%%%%%	%%%%%%%%%%%%%%%%%%%%%%%%%%%%%%%%%%%%%	%%%%%%%%%%%%%%%%%%%%%%%%%%%%%%%%%%%%%	%%%%%%%%%%%%%%%%%%%%%%%%%%%%%%%%%%%%%  	%%%%%%%%%%%%%%%%%%%%%%%%%%%%%%%%%%%%%

	\section{Mathematical setting}
	
	\subsection{Admissible kernels}
	The underlying linear dynamics are governed by the memory kernel $g \in L^1_{\text{loc}}(\mathbb{R}^+)$. Capturing the super-diffusive regime at high frequencies, while accommodating potential crossovers to classical or alternative anomalous diffusion regimes at low frequencies, relies on the analytic properties of its Laplace transform $\hat{g}(\lambda) = \mathcal{L}\{g(t)\}(\lambda)$, whose foundational mathematical framework traces back to the classical theory of Widder \cite{Widder1941}. It is assumed that $g$ belongs to the class of admissible kernels satisfying the following structural hypotheses:
	\begin{itemize} 
		\item \textbf{(H1) Asymptotic sectoriality, regularity, and compatibility:} The function $\hat{g}(\lambda)$ is analytic and non-vanishing in $\mathbb{C} \setminus (-\infty, 0]$. There exists a continuous function $\theta$ satisfying $\inf_{r > 0} \theta(r) > 0$ such that the roots of $\lambda + |\xi|^2 \hat{g}(\lambda) = 0$ are contained within the region $|\arg \lambda| \le \pi - \theta(|\lambda|)$. As $|\lambda| \to \infty$, the sector aperture stabilizes to $\theta(|\lambda|) \to \theta_\infty > 0$, satisfying the high-frequency compatibility condition
		\begin{equation}\label{hyp:compatibility_new} 
			\theta_\infty < \frac{\pi \alpha_\infty}{1+\alpha_\infty}, 
		\end{equation} 
		for some $\alpha_\infty \in (0,1)$. 
		\item \textbf{(H2) High-frequency asymptotics (short time):} There exist constants $C_1, C_2 > 0$ such that, for $|\lambda|$ sufficiently large in $\Sigma_{\theta_\infty}$: 
		\begin{equation}\label{hyp:h2} 
			C_1 |\lambda|^{-\alpha_\infty} \le |\hat{g}(\lambda)| \le C_2 |\lambda|^{-\alpha_\infty}. 
		\end{equation} 
		\item \textbf{(H3) Low-frequency asymptotics (long time):} There exist $\alpha_0 \in [0,1)$ and constants $c_1, c_2 > 0$ such that, for $|\lambda|$ sufficiently small: 
		\[ c_1 |\lambda|^{-\alpha_0} \le |\hat{g}(\lambda)| \le c_2 |\lambda|^{-\alpha_0}. \] 
	\end{itemize}

	To illustrate the physical relevance of the admissible kernel class, four canonical examples of anomalous transport are highlighted, with particular emphasis on their spectral configurations.

	\noindent \textbf{Example 1 (Pure fractional diffusion).} The standard time-fractional model is recovered by choosing the Riemann-Liouville kernel $g(t) = \frac{t^{\alpha-1}}{\Gamma(\alpha)}$ for $\alpha \in (0,1)$, a canonical structure in fractional calculus frameworks (see, e.g., Kilbas et al. \cite{Kilbas2006}). Its Laplace transform is $\hat{g}(\lambda) = \lambda^{-\alpha}$. This kernel satisfies (H1)-(H3) with scale invariance, where the high-frequency and low-frequency asymptotic exponents collapse to a single value given by $\alpha_\infty = \alpha_0 = \alpha$.

	\noindent \textbf{Example 2 (Multi-term fractional diffusion).} In complex viscoelastic media presenting structural heterogeneities, the transport mechanism often transitions between different anomalous regimes over time. These dynamics can be modeled by a sum of fractional kernels of the form
	\begin{equation*}
		g(t) = c_1 \frac{t^{\alpha-1}}{\Gamma(\alpha)} + c_2 \frac{t^{\beta-1}}{\Gamma(\beta)}, \quad \text{with } c_1, c_2 > 0 \text{ and } 0 < \beta < \alpha < 1.
	\end{equation*}
	The Laplace transform is $\hat{g}(\lambda) = c_1 \lambda^{-\alpha} + c_2 \lambda^{-\beta}$. The roots of $\hat{g}(\lambda) = 0$ satisfy $|\arg \lambda| = \frac{\pi}{\alpha-\beta} > \pi$, placing them outside the principal branch, which guarantees sectoriality (H1). As $|\lambda| \to \infty$, the behavior is dominated by $\beta$, yielding $\alpha_\infty = \beta$, whereas as $|\lambda| \to 0$, the $\alpha$ term dominates, leading to $\alpha_0 = \alpha$.

	\noindent \textbf{Example 3 (Fractional retardation and Cole-Cole relaxation).} Physical systems exhibiting a transition from an initial anomalous super-diffusive state to a standard Gaussian diffusion profile at long times are modeled via fractional retardation kernels, widely established in dielectric relaxation and viscoelasticity \cite{ColeCole1941, Mainardi2010}. The algebraic structure in the Laplace domain is given by
	\begin{equation*}
		\hat{g}(\lambda) = \frac{1}{\lambda^\alpha + \gamma}, \quad \text{for } \alpha \in (0,1) \text{ and } \gamma > 0,
	\end{equation*}
	corresponding to the time-domain memory kernel $g(t) = t^{\alpha-1} E_{\alpha, \alpha}(-\gamma t^\alpha)$, where $E_{\alpha,\alpha}$ represents the classical two-parameter Mittag-Leffler function.

	Verifying the asymptotic hypotheses requires evaluating the limits of the transform $\hat{g}(\lambda)$. As $|\lambda| \to \infty$, the scaling $\lambda^\alpha \gg \gamma$ implies $|\hat{g}(\lambda)| \sim |\lambda|^{-\alpha}$, establishing the short-time exponent $\alpha_\infty = \alpha \in (0,1)$, which satisfies (H2). Conversely, as $|\lambda| \to 0$, the term $\lambda^\alpha$ vanishes against the strictly positive constant $\gamma$, yielding $\lim_{\lambda \to 0} \hat{g}(\lambda) = \gamma^{-1}$. The low-frequency transform behaves as $\gamma^{-1}|\lambda|^0$, confirming the long-time classical crossover exponent $\alpha_0 = 0 \in [0,1)$, in agreement with (H3).

	The sectoriality condition (H1) dictates that the roots of the characteristic equation $D(\lambda) = \lambda + |\xi|^2 \hat{g}(\lambda) = 0$ remain bounded away from the negative real axis. Substituting $\hat{g}(\lambda)$ into $D(\lambda) = 0$ and clearing denominators yields the polynomial relation $\lambda^{1+\alpha} + \gamma \lambda + |\xi|^2 = 0$. Testing for the existence of roots along the principal branch cut is performed by mapping $\lambda = r e^{i\pi}$ for $r > 0$. Evaluating the imaginary part of the system yields
	\begin{equation*}
		\Im\left( r^{1+\alpha} e^{i\pi(1+\alpha)} + \gamma r e^{i\pi} + |\xi|^2 \right) = r^{1+\alpha} \sin(\pi + \pi\alpha) = -r^{1+\alpha} \sin(\pi\alpha).
	\end{equation*}
	Since $\alpha \in (0,1)$, it follows that $\sin(\pi\alpha) > 0$. Thus,
	\[ \Im(D(re^{i\pi})) = -r^{1+\alpha}\sin(\pi\alpha) < 0 \]
	for all $r > 0$ and any spatial frequency $|\xi|^2$. Therefore, the imaginary part cannot vanish on the branch cut, and no characteristic roots emerge on $(-\infty, 0]$. Given that $\hat{g}(\lambda)$ is analytic and non-vanishing on $\mathbb{C} \setminus (-\infty, 0]$, Cauchy's integral theorem permits the deformation of the Bromwich contour into a rigid Hankel structure enclosing the origin, establishing (H1) with a uniform margin $\theta(|\lambda|) > 0$.

	\noindent \textbf{Example 4 (Multi-scale Prabhakar memory).} Transport mechanisms in macroscopically disordered media or hierarchical networks often transition between two distinct anomalous regimes before reaching saturation. Such behavior is captured by multi-scale kernels whose framework is governed by the generalized Mittag-Leffler functions of Prabhakar \cite{Giusti2020, Prabhakar1971}. The transform is defined as
	\begin{equation*}
		\hat{g}(\lambda) = \frac{1}{\lambda^{\alpha_0} + \tau \lambda^{\alpha_\infty}}, \quad \text{with } 0 < \alpha_0 < \alpha_\infty < 1 \text{ and } \tau > 0.
	\end{equation*}
	For short times ($|\lambda| \to \infty$), the high-frequency growth is dominated by the higher power $\tau \lambda^{\alpha_\infty}$, yielding the algebraic control $|\hat{g}(\lambda)| \sim \tau^{-1}|\lambda|^{-\alpha_\infty}$, confirming (H2). For long times ($|\lambda| \to 0$), the low-frequency relaxation is dictated by the lower power $\lambda^{\alpha_0}$, leading to $|\hat{g}(\lambda)| \sim |\lambda|^{-\alpha_0}$, confirming (H3).

	To verify hypothesis (H1), the corresponding characteristic equation is written as $\lambda^{1+\alpha_0} + \tau \lambda^{1+\alpha_\infty} + |\xi|^2 = 0$. Testing for the emergence of singular poles on the negative real line via $\lambda = r e^{i\pi}$ ($r > 0$), the imaginary component resolves to
	\begin{equation*}
		\Im\left( r^{1+\alpha_0} e^{i\pi(1+\alpha_0)} + \tau r^{1+\alpha_\infty} e^{i\pi(1+\alpha_\infty)} + |\xi|^2 \right) = -r^{1+\alpha_0}\sin(\pi\alpha_0) - \tau r^{1+\alpha_\infty}\sin(\pi\alpha_\infty).
	\end{equation*}
	As both parameters satisfy $\alpha_0, \alpha_\infty \in (0,1)$, the trigonometric terms $\sin(\pi\alpha_0)$ and $\sin(\pi\alpha_\infty)$ are strictly positive. Their linear combination with negative signs forces 
	\[ \Im(D(re^{i\pi})) < 0 \] 
	for all $r > 0$. This absence of zeros on the negative real axis prevents phase cancellation, ensuring that the resolvent symbol remains uniformly coercive along adaptive contours, thereby satisfying hypothesis (H1).

	\begin{remark}
		The local well-posedness and short-time mapping properties developed in Sections 2 and 3 depend on the high-frequency parameter $\alpha_\infty \in (0,1)$. In contrast, the low-frequency parameter $\alpha_0 \in [0,1)$ encodes the global asymptotic integrability required to establish the existence of global-in-time solutions and their decay rates.
	\end{remark}

	\subsection{Mild framework}
	The mild integral formulation of \eqref{eq:main} follows from the variation of constants formula. Applying the spatial Fourier transform---in the sense of tempered distributions $\mathcal{S}'(\mathbb{R}^N)$---alongside the temporal Laplace transform yields an algebraic relation. The spatial Fourier, temporal Laplace, and mixed Fourier-Laplace transforms are uniformly denoted by $\widehat{\cdot}$ (where the specific integral operator is identifiable by its dual variables $\xi$, $\lambda$, or both). Transforming the governing equation gives
	\begin{equation*}
		\lambda \hat{u}(\xi,\lambda) - \hat{u}_0(\xi) = -|\xi|^2 \hat{g}(\lambda) \hat{u}(\xi,\lambda) + \widehat{f(u)}(\xi,\lambda).
	\end{equation*}
	Algebraic manipulation isolates the state variable $\hat{u}(\xi,\lambda)$, leading to
	\begin{equation}\label{eq:algebraic_state}
		\hat{u}(\xi,\lambda) = \frac{1}{\lambda + |\xi|^2 \hat{g}(\lambda)} \hat{u}_0(\xi) + \frac{1}{\lambda + |\xi|^2 \hat{g}(\lambda)} \widehat{f(u)}(\xi,\lambda).
	\end{equation}
	The resolvent operator $S(t)$ is defined through its Fourier symbol $\widehat{S}(t,\xi)$, which is obtained by the inverse Laplace transform (Bromwich integral) of the transfer function (see, e.g., Pr\"uss \cite[Chapter 1]{Pruss1993}, Arendt et al. \cite{Arendt2011}, and Bajlekova \cite{Bajlekova2001}):
	\begin{equation*}
		\widehat{S}(t,\xi) = \mathcal{L}^{-1} \left\{ \frac{1}{\lambda + |\xi|^2 \hat{g}(\lambda)} \right\}(t) = \frac{1}{2\pi i} \int_{\Gamma} \frac{e^{\lambda t}}{\lambda + |\xi|^2 \hat{g}(\lambda)} d\lambda,
	\end{equation*}
	where $\Gamma$ is a Hankel contour circumventing the negative real axis, as justified by (H1). By the convolution theorem for Laplace transforms, the inverse Laplace operator maps the nonlinear component of \eqref{eq:algebraic_state} to
	\begin{equation*}
		\mathcal{L}^{-1} \left\{ \widehat{S}(\cdot,\xi) \widehat{f(u)}(\cdot,\xi) \right\}(t) = \int_0^t \widehat{S}(t-s,\xi) \mathcal{F}\{f(u(\cdot,s))\}(\xi) ds.
	\end{equation*}
	Applying the inverse Fourier transform to this expression motivates the following definition. 
	\begin{de}
		A mild solution to \eqref{eq:main} is a continuous function $u:[0,T]\to L^{q}(\mathbb{R}^N)$ satisfying the integral equation
		\begin{equation*}
			u(t) = S(t)u_0 + \int_0^t S(t-s)u(s)|u(s)|^{\rho-1} ds.
		\end{equation*}
	\end{de}

	\section{Linear estimates}
	Establishing local well-posedness in $L^q(\mathbb{R}^N)$ relies on extracting the $L^q-L^p$ smoothing properties of the resolvent operator $S(t)$. Because the spatial kernel associated with $S(t)$ lacks global integrability in high dimensions within the super-diffusive regime, the operator bounds are obtained via Fourier multiplier theory, fractional Sobolev embeddings, and scaling arguments.
	
	\begin{lema}[$L^q-L^p$ smoothing estimates]\label{lem:smoothing}
		Let $g$ satisfy hypotheses (H1) and (H2), and let $T_0 > 0$. Assume that $1 \le q \le p < \infty$ satisfies the dimensional restriction
		\begin{equation}\label{eq:dimensional_restriction}
			\frac{1}{q} - \frac{1}{p} < \frac{2}{N}.
		\end{equation}
		Then, there exists a constant $C > 0$, independent of $t \in (0, T_0]$ and $u_0 \in L^q(\mathbb{R}^N)$, such that the resolvent operator satisfies
		\begin{equation*}
			\|S(t)u_0\|_{L^p(\mathbb{R}^N)} \le C t^{-\gamma_{q,p}} \|u_0\|_{L^q(\mathbb{R}^N)},
		\end{equation*}
		where $\gamma_{q,p} = \frac{N(1+\alpha_\infty)}{2}\left(\frac{1}{q} - \frac{1}{p}\right)$.
	\end{lema}
	\begin{proof}
		For $t \in (0, T_0]$, isolating the temporal singularity requires scaling the Laplace variable as $\mu = \lambda t$ within the Bromwich integral, thereby defining the deformation function 
		\[ H(t,\mu) := (\mu/t)^{\alpha_\infty} \hat{g}(\mu/t). \] 
		By choosing $T_0$ sufficiently small, the scaled variable $|\lambda| = |\mu/t|$ lies in the high-frequency domain for any $\mu$ separated from the origin, ensuring the uniform validity of \eqref{hyp:h2}. Consequently, $H(t,\mu)$ remains uniformly bounded above and away from zero. The multiplier scales as $\widehat{S}(t,\xi) = \widetilde{\Psi}(t,\eta)$, with $\eta = \xi t^{\frac{1+\alpha_\infty}{2}}$ and
		\begin{equation*}
			\widetilde{\Psi}(t,\eta) = \frac{1}{2\pi i} \int_{\Gamma'} \frac{e^\mu}{\mu + |\eta|^2 \mu^{-\alpha_\infty} H(t,\mu)} d\mu.
		\end{equation*}
		
		Since the symbol exhibits distinct asymptotic regimes near the origin and at high frequencies, partitioning the frequency space is necessary. Introducing a smooth cut-off $\chi \in C_0^\infty(\mathbb{R}^N)$ with $\chi(\eta) = 1$ for $|\eta| \le 1$ and $\chi(\eta) = 0$ for $|\eta| \ge 2$, the symbol decomposes as 
		\[ \widetilde{\Psi} = \Psi_{\mathrm{low}} + \Psi_{\mathrm{high}}, \] 
		where $\Psi_{\mathrm{low}} = \chi \widetilde{\Psi}$ and $\Psi_{\mathrm{high}} = (1-\chi) \widetilde{\Psi}$.
		
		\noindent\textbf{Step 1: Low-frequency regime ($\Psi_{\mathrm{low}}$).} 
		For $|\eta| \le 2$, the roots of the denominator satisfy $|\mu_k| \sim |\eta|^{\frac{2}{1+\alpha_\infty}}$ and remain confined to a bounded region in the open left half-plane. Cauchy's Theorem allows the integral to be anchored along a fixed Hankel path 
		\[ \Gamma_1 = \{ \delta + r e^{\pm i\phi} : r \ge 0 \}. \] 
		Choosing $\delta > 0$ and an angle $\phi \in (\pi/2, \pi)$ ensures this contour encloses the origin and all potential roots, while remaining within the domain of analyticity off the principal branch cut as guaranteed by (H1). Along $\Gamma_1$, the denominator is uniformly coercive. The exponential decay $e^{r\cos\phi}$ allows for infinite differentiation under the integral sign, yielding $\widetilde{\Psi}(t,\cdot) \in C^\infty$ for $|\eta| \le 2$. Thus, the truncated symbol $\Psi_{\mathrm{low}}(t,\cdot)$ belongs uniformly to the Schwartz space $\mathcal{S}(\mathbb{R}^N)$, producing a bounded $L^q \to L^p$ operator.
		
		\noindent\textbf{Step 2: High-frequency regime ($\Psi_{\mathrm{high}}$).} 
		For $|\eta| \ge 1$, the dynamics are governed by the moving poles. The compatibility condition \eqref{hyp:compatibility_new} ensures that the asymptotic phases of the roots, $\pm \frac{\pi}{1+\alpha_\infty}$, lie within the sector of aperture $\pi - \theta_\infty$. Deforming the contour into the left half-plane without leaving the domain of analyticity splits the integral into a residue component ($\widetilde{\Psi}_{\mathrm{res}}$) and a branch cut component ($\widetilde{\Psi}_{\mathrm{cut}}$). The residue contribution, evaluated at the poles $\mu_k$, satisfies
		\begin{equation*}
			|(1-\chi(\eta)) \widetilde{\Psi}_{\mathrm{res}}(t, \eta)| \le C_1 \exp\left( -c |\eta|^{\frac{2}{1+\alpha_\infty}} \right), \quad c > 0.
		\end{equation*}
		This rapid exponential decay dominates any polynomial growth, implying that the corresponding spatial kernel belongs uniformly to the Schwartz space $\mathcal{S}(\mathbb{R}^N)$, thereby mapping $L^q \to L^p$ boundedly for any $1 \le q \le p \le \infty$.
		
		Evaluating the branch cut contribution $\widetilde{\Psi}_{\mathrm{cut}}$ and establishing the uniform $L^q \to L^p$ bounding properties of this high-frequency remainder require ensuring that the deformed contour maintains a non-degenerate lower bound against the moving poles. Let the characteristic denominator be denoted by
		\begin{equation}\label{eq:app_den}
			D(\mu, \eta) = \mu + |\eta|^2 \mu^{-\alpha_\infty} H(t,\mu),
		\end{equation}
		where the evaluation occurs along the rays $\mu = r e^{\pm i \phi_c}$. The high-frequency compatibility condition \eqref{hyp:compatibility_new} restricts the sectorial aperture to $\theta_\infty < \frac{\pi \alpha_\infty}{1+\alpha_\infty}$, permitting the choice of a cutting phase $\phi_c$ such that
		\begin{equation}\label{eq:phase_restriction}
			\frac{\pi}{1+\alpha_\infty} < \phi_c < \pi.
		\end{equation}
		
		Characterizing the phase interaction within the non-local symbol prevents destructive cancellation in supercritical regimes. Letting $m(t,\mu) = \mu^{-\alpha_\infty} H(t,\mu)$, hypothesis (H2) dictates that as $t \to 0^+$ and $|\mu/t| \to \infty$, the modulation function converges to $C_1 \in \mathbb{R}^+$; consequently, the asymptotic phase angle of $m(t,\mu)$ reduces to $\theta_m \approx -\alpha_\infty \phi_c$. The relative phase separation between the linear transport component $\mu$ and the memory amplification term $|\eta|^2 m(t,\mu)$ is given by
		\begin{equation*}
			\Phi = \arg(\mu) - \arg(m(t,\mu)) \approx \phi_c - (-\alpha_\infty \phi_c) = (1+\alpha_\infty)\phi_c.
		\end{equation*}
		Enforcing the boundary condition \eqref{eq:phase_restriction} ensures the phase parameter $\Phi$ satisfies the geometric constraint $\pi < \Phi < (1+\alpha_\infty)\pi < 2\pi$. Given that $\Phi$ is bounded away from the boundaries of the interval, there exists a structural constant $\delta_0 = \delta_0(\alpha_\infty, \phi_c) > 0$ such that $\cos(\Phi) \ge -1 + \delta_0 > -1$ for all $\mu$ along the branch cut contour.
		
		Applying the law of cosines to the algebraic sum in \eqref{eq:app_den} establishes the uniform coercivity of the symbol denominator, yielding
		\begin{equation}\label{eq:cosines_expansion}
			|D(\mu, \eta)|^2 = r^2 + |\eta|^4 |m|^2 + 2 r |\eta|^2 |m| \cos(\Phi).
		\end{equation}
		Substituting the uniform lower bound for $\cos(\Phi)$ into \eqref{eq:cosines_expansion} provides
		\begin{align*}
			|D(\mu, \eta)|^2 &\ge r^2 + |\eta|^4 |m|^2 - 2(1-\delta_0) r |\eta|^2 |m| \\
			&= (1-\delta_0) \left( r - |\eta|^2 |m| \right)^2 + \delta_0 \left( r^2 + |\eta|^4 |m|^2 \right).
		\end{align*}
		Since the first quadratic term is non-negative, dropping it preserves the inequality. Taking the square root of both sides and applying the elementary algebraic inequality $\sqrt{a^2 + b^2} \ge \frac{1}{\sqrt{2}}(a+b)$ isolates the standard coercive limit
		\begin{equation*}
			|D(\mu, \eta)| \ge \sqrt{\frac{\delta_0}{2}} \left( r + |\eta|^2 |m| \right) \ge C_{\alpha_\infty} \left( r + C_1 |\eta|^2 r^{-\alpha_\infty} \right).
		\end{equation*}
		
		The uniformity of the prefactor $C_{\alpha_\infty} > 0$ confirms that the symbol denominator never experiences scale degeneration. With this lower bound secured, the spatial derivatives are evaluated. The first-order gradient is given by $\nabla_\eta D = 2\eta m(t,\mu)$. Applying the coercive bound $|D| \ge C_{\alpha_\infty} C_1 |\eta|^2 |m|$ leads directly to
		\begin{equation*}
			|\nabla_\eta D| \le 2|\eta| |m| = 2|\eta|^{-1} (|\eta|^2 |m|) \le C |\eta|^{-1} |D|.
		\end{equation*}
		Consequently, the first derivative of the inverse denominator satisfies
		\begin{equation*}
			|\nabla_\eta (D^{-1})| = |-D^{-2} \nabla_\eta D| \le C |D|^{-2} |\eta|^{-1} |D| = C |D|^{-1} |\eta|^{-1}.
		\end{equation*}
		Proceeding by induction and applying Fa\`a di Bruno's formula for higher-order compositions, each subsequent spatial derivative extracts an additional factor of $|\eta|^{-1}$ while preserving the structural denominator $|D|^{-1}$. Since the evaluation is restricted to the support $|\eta| \ge 1$, we obtain the generalized bound
		\begin{equation*}
			|\nabla_\eta^\beta (D^{-1})| \le C_\beta |D|^{-1} |\eta|^{-|\beta|}.
		\end{equation*}
		
		Integrating this expression against the exponential factor $e^{-r |\cos \phi_c|} dr$ along the cut yields
		\begin{equation*}
			|\nabla_\eta^\beta \widetilde{\Psi}_{\mathrm{cut}}(t, \eta)| \le C_\beta \int_0^\infty \frac{e^{-r|\cos \phi_c|}}{|D|} |\eta|^{-|\beta|} dr.
		\end{equation*}
		Extracting the algebraic decay requires utilizing the coercive bound $|D| \ge C_{\alpha_\infty}(r + C_1 |\eta|^2 r^{-\alpha_\infty})$ and partitioning the integration domain. For the principal term ($|\beta| = 0$), we obtain
		\begin{align*}
			\int_0^\infty \frac{e^{-r|\cos \phi_c|}}{r + C_1 |\eta|^2 r^{-\alpha_\infty}} dr &\le \int_0^1 \frac{1}{C_1 |\eta|^2 r^{-\alpha_\infty}} dr + \int_1^\infty \frac{e^{-r|\cos \phi_c|}}{C_1 |\eta|^2 r^{-\alpha_\infty}} dr \\
			&= \frac{1}{C_1 |\eta|^2} \left( \int_0^1 r^{\alpha_\infty} dr + \int_1^\infty r^{\alpha_\infty} e^{-r|\cos \phi_c|} dr \right).
		\end{align*}
		As $\alpha_\infty \in (0,1)$, the first integral converges exactly to $(1+\alpha_\infty)^{-1}$. The second integral is bounded by $C_{\phi_c} \Gamma(1+\alpha_\infty)$, where $C_{\phi_c}$ depends exclusively on the fixed phase $\phi_c$. Since these evaluations are independent of the parameter transitions inherent to the multi-scale kernel, the structural constant remains uniformly bounded. This establishes that the full integration is bounded by $C' |\eta|^{-2}$. Factoring in the general derivative weight $|\eta|^{-|\beta|}$ and bounding $|\eta|^{-|\beta|} \le C (1+|\eta|)^{-|\beta|}$ across the high-frequency domain produces the required algebraic decay
		\begin{equation*}
			|\nabla_\eta^\beta (\Psi_{\mathrm{high}})(t, \eta)| \le C_\beta'' (1 + |\eta|)^{-2-|\beta|}.
		\end{equation*}
		
		This estimate demonstrates that $\Psi_{\mathrm{high}}(t, \cdot)$ belongs uniformly to the H\"ormander symbol class $S^{-2}_{1,0}$. Let $T_t$ be the associated pseudo-differential operator. For $q > 1$, the Mikhlin-H\"ormander multiplier theorem ensures $T_t$ maps $L^q(\mathbb{R}^N) \to W^{2,q}(\mathbb{R}^N)$. Applying the continuous Sobolev embedding $W^{2,q}(\mathbb{R}^N) \hookrightarrow L^p(\mathbb{R}^N)$---enforcing the dimensional constraint \eqref{eq:dimensional_restriction}---it follows that $T_t$ acts as a bounded linear map from $L^q(\mathbb{R}^N)$ to $L^p(\mathbb{R}^N)$.
		
		Because the continuous Sobolev embedding fails at the critical endpoint $q=1$, analyzing the spatial convolution kernel associated with the high-frequency multiplier becomes necessary. Defined by 
		\[\mathcal{G}_t(w) := \mathcal{F}^{-1}(\Psi_{\mathrm{high}}(t, \cdot))(w),\] 
		the symbol's membership in $S^{-2}_{1,0}$ dictates via classical pseudo-differential theory that $\mathcal{G}_t(w)$ inherits explicit asymptotic properties (detailed in Remark \ref{rem:kernel_singularity}). Specifically, $\mathcal{G}_t$ exhibits a localized singularity of order $\mathcal{O}(|w|^{-(N-2)})$ near the origin for $N \ge 3$ (and a logarithmic singularity for $N=2$), while maintaining rapid polynomial decay at infinity. This property guarantees that $\mathcal{G}_t \in L^r(\mathbb{R}^N)$ for any $1 \le r < \frac{N}{N-2}$. Applying Young's convolution inequality to $T_t v = \mathcal{G}_t \ast v$ requires the kernel to be in $L^r(\mathbb{R}^N)$ with $1 + \frac{1}{p} = \frac{1}{1} + \frac{1}{r}$, which simplifies directly to $r = p$. The required integrability $p < \frac{N}{N-2}$ recovers the endpoint dimensional restriction \eqref{eq:dimensional_restriction}, namely $1 - \frac{1}{p} < \frac{2}{N}$. This yields the uniform bound $\|T_t v\|_{L^p} \le C \|v\|_{L^1}$ without topological degeneration.
		
		\noindent\textbf{Step 3: Temporal scaling.} 
		With the spatial regularizing properties established for the unscaled multiplier $\widetilde{\Psi}(t,\eta)$, the temporal decay is recovered via spatial dilation. Let $T_t$ denote the bounded pseudo-differential operator associated with the symbol $\widetilde{\Psi}(t,\eta)$, acting on the scaled frequency variable $\eta$. Setting the temporal dilation parameter $\kappa = t^{\frac{1+\alpha_\infty}{2}}$, the spatial frequency satisfies $\eta = \kappa \xi$. The full resolvent then acts as 
		\[S(t)u_0(x) = \mathcal{F}^{-1}(\widetilde{\Psi}(t, \kappa \xi) \hat{u}_0(\xi))(x).\]
		
		Defining the dilated initial datum $u_{0,\kappa}(x) = u_0(\kappa x)$ gives $\widehat{u_{0,\kappa}}(\xi) = \kappa^{-N} \hat{u}_0(\kappa^{-1}\xi)$. The action of the unscaled operator $T_t$ on this dilated profile yields
		\begin{equation*}
			T_t(u_{0,\kappa})(y) = \frac{1}{(2\pi)^N} \int_{\mathbb{R}^N} e^{iy \cdot \eta} \widetilde{\Psi}(t, \eta) \kappa^{-N} \hat{u}_0(\kappa^{-1}\eta) d\eta = S(t)u_0(y\kappa).
		\end{equation*}
		Taking the $L^p$ norm of both sides leads to
		\begin{equation*}
			\|S(t)u_0\|_{L^p}^p = \int_{\mathbb{R}^N} |T_t(u_{0,\kappa})(\kappa^{-1} x)|^p dx = \kappa^N \|T_t(u_{0,\kappa})\|_{L^p}^p.
		\end{equation*}
		Combining this identity with the uniform operator bound $\|T_t(u_{0,\kappa})\|_{L^p} \le C \|u_{0,\kappa}\|_{L^q}$ and the intrinsic scaling $\|u_{0,\kappa}\|_{L^q} = \kappa^{-N/q} \|u_0\|_{L^q}$ yields
		\begin{equation*}
			\|S(t)u_0\|_{L^p} = \kappa^{N/p} \|T_t(u_{0,\kappa})\|_{L^p} \le C \kappa^{N/p} \kappa^{-N/q} \|u_0\|_{L^q}.
		\end{equation*}
		Substituting the scaling parameter $\kappa = t^{\frac{1+\alpha_\infty}{2}}$ produces the required temporal decay
		\begin{equation*}
			\|S(t)u_0\|_{L^p} \le C t^{-\frac{N(1+\alpha_\infty)}{2}\left(\frac{1}{q} - \frac{1}{p}\right)} \|u_0\|_{L^q},
		\end{equation*}
		which concludes the proof.
	\end{proof}
	
	\begin{remark}[Sobolev integrability barrier and classical parabolic theory]
		The dimensional restriction \eqref{eq:dimensional_restriction} represents a fundamental departure from classical parabolic theory. For the standard heat equation $\partial_t u = \Delta u$, the Gaussian symbol exhibits exponential decay, yielding infinite regularizing capacity independent of the dimension $N$. In contrast, the super-diffusive memory operator introduces an algebraic tail at high frequencies (dictated by $\widetilde{\Psi}_{\mathrm{cut}}$), limiting the regularization to a gain of exactly two spatial derivatives. Consequently, the gain in Lebesgue integrability is governed by the continuous Sobolev embedding $W^{2,q}(\mathbb{R}^N) \hookrightarrow L^p(\mathbb{R}^N)$. While classical Sobolev theory permits the critical equality $\frac{1}{q} - \frac{1}{p} = \frac{2}{N}$, the super-diffusive nature of our operator imposes the integrability barrier
		\[ \frac{1}{q} - \frac{1}{p} < \frac{2}{N} \]
		to prevent non-integrable singularities in the Bessel potential.
	\end{remark}
	
	\begin{remark}\label{rem:kernel_singularity}
		From the classical theory of pseudo-differential operators, a multiplier belonging to the H\"ormander class $S^{-2}_{1,0}$ determines a spatial convolution kernel $\mathcal{G}_t(w)$ which is smooth away from the origin and exhibits rapid polynomial decay at infinity. Near the origin, however, it encapsulates a localized singularity governed by the dimension of the space. Specifically, for $N \ge 3$, the kernel behaves asymptotically as $\mathcal{O}(|w|^{-(N-2)})$, yielding the uniform pointwise constraint
		\begin{equation}\label{eq:pointwise_kernel_N3}
			\sup_{w \in \mathbb{R}^N} |w|^{N-2} (1+|w|)^{M} |\mathcal{G}_t(w)| \le C_{\mathcal{G}} < \infty,
		\end{equation}
		for an arbitrarily large decay exponent $M > 0$. Alternatively, in the critical two-dimensional geometry $N = 2$, the local singularity is characteristically logarithmic, satisfying
		\begin{equation}\label{eq:pointwise_kernel_N2}
			\sup_{w \in \mathbb{R}^N} \left[ \ln\left(1 + \frac{2}{|w|}\right) \right]^{-1} (1+|w|)^{M} |\mathcal{G}_t(w)| \le C_{\mathcal{G}} < \infty.
		\end{equation}
		Conversely, in the one-dimensional case $N = 1$, the algebraic decay of the symbol of order $-2$ is sufficient to guarantee global integrability in frequency space. Consequently, by the Riemann-Lebesgue lemma, the spatial kernel $\mathcal{G}_t(w)$ is globally bounded and continuous, possessing no singularity at the origin and satisfying
		\begin{equation}\label{eq:pointwise_kernel_N1}
			\sup_{w \in \mathbb{R}} (1+|w|)^{M} |\mathcal{G}_t(w)| \le C_{\mathcal{G}} < \infty,
		\end{equation}
		where the structural prefactors remain uniformly bounded across transient time scales.
	\end{remark}

	\begin{lema}[Asymptotic $L^m-L^p$ smoothing estimates]\label{lem:smoothing_asymptotic}
		Let $g$ satisfy hypotheses (H1), (H2), and (H3). Assume that $1 \le m \le p < \infty$ satisfies the dimensional restriction
		\begin{equation}\label{eq:dimensional_restriction_asymptotic}
			\frac{1}{m} - \frac{1}{p} < \frac{2}{N}.
		\end{equation}
		Then, there exists a constant $C > 0$, independent of $t \ge 1$ and $u_0 \in L^m(\mathbb{R}^N)$, such that the resolvent operator satisfies
		\begin{equation*}
			\|S(t)u_0\|_{L^p(\mathbb{R}^N)} \le C t^{-\beta_{m,p}} \|u_0\|_{L^m(\mathbb{R}^N)},
		\end{equation*}
		where $\beta_{m,p} = \frac{N(1+\alpha_0)}{2}\left(\frac{1}{m} - \frac{1}{p}\right)$.
	\end{lema}
	\begin{proof}
		For $t \ge 1$, a smooth dyadic partition of unity $1 = \psi_0(\xi) + \sum_{j=1}^\infty \psi_j(\xi)$ isolates the active frequency scales. The support of $\psi_0 \in C^\infty(\mathbb{R}^N)$ covers the high-frequency range $|\xi| \ge 1$, whereas each $\psi_j(\xi) = \psi(2^j \xi)$ isolates the low-frequency dyadic annuli $|\xi| \sim 2^{-j}$ for $j \ge 1$. The resolvent operator decomposes accordingly into $S(t) = S_0(t) + \sum_{j=1}^\infty S_j(t)$.
		
		\noindent\textbf{Step 1: High-frequency branch ($S_0$).} 
		For $|\xi| \ge 1$, the roots $\lambda_k(\xi)$ of $\lambda + |\xi|^2 \hat{g}(\lambda) = 0$ are bounded away from the origin and confined within the analytic sector. By hypothesis (H1), their real parts satisfy $\mathrm{Re}(\lambda_k) \le -c_0 |\xi|^{\frac{2}{1+\alpha_\infty}} \le -c_0 < 0$. Deforming the Bromwich contour to the boundary rays of $\Sigma_{\theta_\infty}$, the local contribution near the origin, controlled via (H3), is of order $\mathcal{O}(t^{-(1+\alpha_0)} |\xi|^{-2})$, while the outer integral yields an exponentially decaying tail $\mathcal{O}(e^{-\tilde{c} t})$. Applying the Mikhlin-H\"ormander theorem for $m > 1$, or a refined Bessel-potential decay estimate for the endpoint $m=1$, produces
		\begin{equation*}
			\|S_0(t)u_0\|_{L^p} \le C t^{-(1+\alpha_0)} \|u_0\|_{L^m}, \quad \forall t \ge 1.
		\end{equation*}
		Due to the dimensional restriction \eqref{eq:dimensional_restriction_asymptotic}, the exponent satisfies $\beta_{m,p} < 1+\alpha_0$. Since the evaluation occurs in the long-time regime $t \ge 1$, the faster decay is absorbed by the slower one, implying $t^{-(1+\alpha_0)} \le t^{-\beta_{m,p}}$ uniformly. This structural bound ensures that the high-frequency branch remains subordinate to the target decay rate.
		
		\noindent\textbf{Step 2: Low-frequency dyadic multipliers and adaptive contours ($S_j$).} 
		For the low-frequency components $j \ge 1$, the spatial frequencies are localized on the annuli $|\xi| \sim 2^{-j}$. Hypothesis (H3) dictates that the dominant poles migrate towards the origin according to the scaling law $\lambda_j \sim -|\xi|^{\gamma_0} \sim -2^{-j \gamma_0}$, where the characteristic exponent is defined as $\gamma_0 := \frac{2}{1+\alpha_0}$. 
		
		To prevent the singularity distance from collapsing and to ensure uniform coercivity, a family of adaptive Hankel contours $\Gamma_j$ is constructed for each dyadic block. The vertex of $\Gamma_j$ on the real axis scales with $2^{-j \gamma_0}$, ensuring that the distance $d_j$ from the contour to the moving pole scales proportionally to the magnitude of the pole ($d_j \sim |\lambda_j|$). 
		
		Defining the localized variables $\mu = 2^{j \gamma_0} \lambda$ and $\eta = 2^j \xi$ maps the domain to a fixed annulus $|\eta| \sim 1$, where the singularity is situated at $\mathcal{O}(1)$ relative to the rescaled contour. The derivatives of the localized symbol $\widehat{S}_j(t,\xi) = \psi_j(\xi)\widetilde{\Psi}(t,\xi)$ respect the scaling of the dyadic projection. Specifically, each spatial derivative $\nabla_\xi$ extracts a scaling factor of $2^j$. Applying Fa\`a di Bruno's formula to the denominator $D(\mu, \eta)$ generates combined spatial terms of the form $2^{jk} \nabla_\eta^k (D^{-1})$. As the distance to the moving poles scales as $d_j \sim 2^{-j \gamma_0}$, the singular algebraic growth of the inverted powers $1/D^{1+k}$ is neutralized by the geometric decay of the dyadic multipliers evaluated along the scaled contour, yielding the uniform H\"ormander condition
		\begin{equation*}
			\sup_{j \ge 1} \sup_{\xi \in \operatorname{supp}(\psi_j)} |\xi|^k |\nabla_\xi^k \widehat{S}_j(t,\xi)| \le C_k e^{-c t 2^{-j \gamma_0}},
		\end{equation*}
		where the constants $C_k$ and $c > 0$ are independent of the dyadic index $j \ge 1$ and time $t \ge 1$.
		
		\noindent\textbf{Step 3: Besov embedding and dyadic summation.} 
		By classical singular integral theory (see Stein \cite{Stein1970}), the uniform H\"ormander multiplier bounds imply that the underlying convolution kernels possess rapid spatial decay and belong uniformly to $L^1(\mathbb{R}^N)$. Combining Young's convolution inequality with standard Bernstein estimates for functions localized in frequency annuli of radius $R \sim 2^{-j}$, the localized operators satisfy
		\begin{equation}\label{eq:bernstein_applied}
			\|S_j(t)u_0\|_{L^p} \le C 2^{-j N \left(\frac{1}{m} - \frac{1}{p}\right)} e^{-c t 2^{-j \gamma_0}} \|\Delta_j u_0\|_{L^m},
		\end{equation}
		where $\Delta_j$ denotes the standard Littlewood-Paley projection (cf. Grafakos \cite[Chapter 6]{Grafakos2009}). Utilizing the canonical continuous embedding $L^m(\mathbb{R}^N) \hookrightarrow \dot{B}^0_{m, \infty}(\mathbb{R}^N)$, the block norm is bounded uniformly by $\|\Delta_j u_0\|_{L^m} \le C \|u_0\|_{L^m}$.
		
		To perform the summation over $j \ge 1$ for $t \ge 1$, the critical time-dependent transition index $j_0(t) \in \mathbb{N}$ is introduced, defined such that $t 2^{-j_0(t) \gamma_0} \sim 1$, which implies $2^{-j_0(t)} \sim t^{-1/\gamma_0} = t^{-\frac{1+\alpha_0}{2}}$. The analysis is then separated into the strict smoothing regime ($m < p$) and the uniform stability configuration ($m = p \ge 1$).
		
		\noindent\textit{Case 1: Strict smoothing ($m < p$).} The summation is partitioned into two sub-regimes
		\begin{equation*}
			\sum_{j=1}^\infty \|S_j(t)u_0\|_{L^p} = \sum_{j=1}^{j_0(t)} \|S_j(t)u_0\|_{L^p} + \sum_{j=j_0(t)+1}^\infty \|S_j(t)u_0\|_{L^p} := I_1 + I_2.
		\end{equation*}
		
		For $I_1$ (damped frequency scales), $j \le j_0(t)$ implies $t 2^{-j \gamma_0} \ge 1$. Given that the exponential damping dominates the algebraic growth for decreasing values of $j$, the sum is controlled by its largest term at the cutoff boundary $j = j_0(t)$:
		\begin{equation*}
			I_1 \le C \|u_0\|_{L^m} \sum_{j=1}^{j_0(t)} e^{-c t 2^{-j \gamma_0}} 2^{-j N \left(\frac{1}{m} - \frac{1}{p}\right)} \le C 2^{-j_0(t) N \left(\frac{1}{m} - \frac{1}{p}\right)} \|u_0\|_{L^m} \le C t^{-\beta_{m,p}} \|u_0\|_{L^m}.
		\end{equation*}
		
		For $I_2$ (ultra-low frequency scales), $j > j_0(t)$ implies $t 2^{-j \gamma_0} < 1$, rendering the exponential term of order $\mathcal{O}(1)$. As the common ratio satisfies $2^{-N(1/m - 1/p)} < 1$, the expression transforms into a convergent geometric series bounded by its first term $j = j_0(t)+1$:
		\begin{equation*}
			I_2 \le C \|u_0\|_{L^m} \sum_{j=j_0(t)+1}^\infty 2^{-j N \left(\frac{1}{m} - \frac{1}{p}\right)} \le C 2^{-j_0(t) N \left(\frac{1}{m} - \frac{1}{p}\right)} \|u_0\|_{L^m} \le C t^{-\beta_{m,p}} \|u_0\|_{L^m}.
		\end{equation*}
		
		\noindent\textit{Case 2: Uniform stability ($m = p \ge 1$).} For $m=p$, the target exponent yields $\beta_{m,m} = 0$, and the classical Mikhlin-H\"ormander theorem cannot be applied at the endpoint $m=1$ due to the failure of $L^1 \to L^1$ boundedness for Calder\'on-Zygmund singular integrals. To circumvent this limitation and secure uniform stability for all $m \ge 1$, the low-frequency resolvent is split at the index $j_0(t)$ as $S_{\mathrm{low}}(t) = \sum_{j=1}^{j_0(t)} S_j(t) + \widetilde{S}_{j_0(t)}(t)$, where $\widetilde{S}_{j_0(t)}(t)$ represents the aggregate low-pass remainder covering the ultra-low frequency ball $|\xi| \le 2^{-j_0(t)} \sim t^{-1/\gamma_0}$.
		
		The sum of the dyadically damped terms is bounded directly via \eqref{eq:bernstein_applied} with $m=p$:
		\begin{equation*}
			\sum_{j=1}^{j_0(t)} \|S_j(t)u_0\|_{L^m} \le C \sum_{j=1}^{j_0(t)} e^{-c t 2^{-j \gamma_0}} \|\Delta_j u_0\|_{L^m} \le C \|u_0\|_{L^m} \sum_{j=1}^{j_0(t)} e^{-c 2^{(j_0(t)-j)\gamma_0}} \le C \|u_0\|_{L^m},
		\end{equation*}
		where convergence is guaranteed by the geometric decay of the exponents in terms of $(j_0(t)-j)$, noting that $\sup_j \|\Delta_j u_0\|_{L^m} \le C \|u_0\|_{L^m}$ uniformly for all $m \ge 1$.
		
		For the ultra-low pass remainder $\widetilde{S}_{j_0(t)}(t)$, let $\chi_0 \in C_0^\infty(\mathbb{R}^N)$ denote the smooth low-pass filter associated with this terminal dyadic sum, such that $\chi_0(\xi) = 1$ for $|\xi| \le 1$ and $\chi_0(\xi) = 0$ for $|\xi| \ge 2$. Its spatial multiplier is given by $\chi_0(2^{j_0(t)}\xi)\widehat{S}(t,\xi)$. Applying the exact temporal similarity scaling $\eta = \xi t^{1/\gamma_0}$ transforms the dynamic cutoff into the static cutoff $\chi_0(\eta)$, anchoring the multiplier to a fixed, $t$-independent compact domain $\operatorname{supp}(\chi_0) \subset \{|\eta| \le 2\}$. Under this scaling, the underlying symbol is formally defined as $\widetilde{S}_t(\eta) := \widehat{S}(t, \eta t^{-1/\gamma_0})$. The inverse Fourier transform yields the rescaled convolution kernel
		\begin{equation*}
			K_{\mathrm{low}}(t,x) = \frac{1}{(2\pi)^N} \int_{\mathbb{R}^N} e^{ix \cdot \xi} \chi_0(\xi t^{1/\gamma_0}) \widehat{S}(t,\xi) d\xi = t^{-N/\gamma_0} \widetilde{K}_t(x t^{-1/\gamma_0}),
		\end{equation*}
		where the rescaled base profile is $\widetilde{K}_t = \mathcal{F}^{-1}(\chi_0(\eta) \widetilde{S}_t(\eta))$. Given that the roots of the characteristic equation under this low-frequency scaling satisfy a uniform coercivity bound analogous to Step 2 of Lemma \ref{lem:smoothing}, the localized symbol $\chi_0(\eta) \widetilde{S}_t(\eta)$ is uniformly smooth with respect to $\eta$. Consequently, its inverse transform $\widetilde{K}_t$ belongs uniformly to the Schwartz class $\mathcal{S}(\mathbb{R}^N)$ for all $t \ge 1$. A direct change of variables $y = x t^{-1/\gamma_0}$ in the spatial integration confirms the exact $L^1(\mathbb{R}^N)$ isometric invariance
		\begin{equation*}
			\|K_{\mathrm{low}}(t,\cdot)\|_{L^1(\mathbb{R}^N)} = \int_{\mathbb{R}^N} t^{-N/\gamma_0} |\widetilde{K}_t(x t^{-1/\gamma_0})| dx = \int_{\mathbb{R}^N} |\widetilde{K}_t(y)| dy \le C,
		\end{equation*}
		uniformly for all $t \ge 1$. Applying Young's convolution inequality bypasses the Calder\'on-Zygmund singularity, yielding
		\begin{equation*}
			\|\widetilde{S}_{j_0(t)}(t)u_0\|_{L^m} \le \|K_{\mathrm{low}}(t,\cdot)\|_{L^1} \|u_0\|_{L^m} \le C \|u_0\|_{L^m}, \quad \forall m \ge 1.
		\end{equation*}
		
		Combining the independent estimates for $S_0(t)$ and $S_{\mathrm{low}}(t)$ establishes the target bound and concludes the proof.
	\end{proof}
	
	\begin{remark}[Temporal H\"older continuity]\label{rem:holder_continuity}
		As a consequence of the multiplier estimates in Lemma \ref{lem:smoothing} and Lemma \ref{lem:smoothing_asymptotic}, the resolvent operator $S(t)$ exhibits local H\"older continuity in time for $t>0$. Taking the temporal derivative of the Bromwich integral for $t>0$ requires subtracting the identity to ensure absolute convergence along the contour, which introduces the multiplier $\lambda$. Upon scaling, this extracts a factor of $t^{-1}$, yielding the derivative bound
		\begin{equation*}
			\|\partial_t S(t)\|_{L^r \to L^p} \le C t^{-1-\gamma_{r,p}}.
		\end{equation*}
		By the Fundamental Theorem of Calculus, this provides the Lipschitz bound 
		\[\|S(t+h) - S(t)\|_{L^r \to L^p} \le C h t^{-1-\gamma_{r,p}}\] 
		for any increment $h>0$. Interpolating this bound with the standard triangle inequality estimate $\|S(t+h) - S(t)\|_{L^r \to L^p} \le 2C t^{-\gamma_{r,p}}$ via a parameter $\theta \in (0,1)$ directly yields the H\"older estimate
		\begin{equation*}
			\|S(t+h) - S(t)\|_{L^r \to L^p} \le C h^\theta t^{-(\gamma_{r,p} + \theta)}, \quad \forall t, h > 0.
		\end{equation*}
		Setting $\tilde{\gamma} = \gamma_{r,p} + \theta$ establishes the uniform control required for the strong temporal continuity of the nonlinear Volterra integrations (analogous optimal regularity results for analytic semigroups are detailed in Lunardi \cite{Lunardi1995}).
	\end{remark}
	
	\section{Local well-posedness and the critical threshold}
	Having established the linear smoothing estimates under the topological constraint \eqref{eq:dimensional_restriction}, the nonlinear Cauchy problem \eqref{eq:main} is addressed via a fixed-point argument. Because the exponent $\gamma_{q,p} = \frac{N(1+\alpha_\infty)}{2}\left(\frac{1}{q} - \frac{1}{p}\right)$ induces a temporal singularity near the origin, solutions must be constructed within a time-weighted Banach space.
	
	\subsection{Local well-posedness in the critical and subcritical regimes}
	For $T > 0$ and $1 < q \le p < \infty$, the resolution space is defined as
	\begin{equation*}
		E_T = \left\{ u \in C([0,T]; L^q(\mathbb{R}^N)) : \sup_{0 < t \le T} t^{\gamma_{q,p}} \|u(t)\|_{L^p} < \infty, \text{ and } \lim_{t \to 0^+} t^{\gamma_{q,p}} \|u(t)\|_{L^p} = 0 \right\},
	\end{equation*}
	equipped with the norm
	\begin{equation*}
		\|u\|_{E_T} = \sup_{0 \le t \le T} \|u(t)\|_{L^q} + \sup_{0 < t \le T} t^{\gamma_{q,p}} \|u(t)\|_{L^p}.
	\end{equation*}
	The closed subspace defined by the vanishing condition as $t \to 0^+$ ensures that the linear evolution $S(t)u_0$ belongs to $E_T$ for any $u_0 \in L^q(\mathbb{R}^N)$. 
	
	Mild solutions are constructed as fixed points of the integral operator
	\begin{equation*}
		\mathcal{T}(u)(t) = S(t)u_0 + \int_0^t S(t-s)u(s)|u(s)|^{\rho-1} ds, \quad t\ge0.
	\end{equation*}
	Applying the linear smoothing estimates to the nonlinear Volterra term yields a fractional integration problem, whose convergence imposes an algebraic balance among the memory exponent $\alpha_\infty$, the spatial dimension $N$, and the reaction power $\rho$, ultimately dictating the critical regularity threshold $q_c$.
	
	\begin{theorem}\label{thm:local_existence}
		Let $g$ satisfy (H1)-(H2). Assume the initial datum $u_0 \in L^q(\mathbb{R}^N)$ and the nonlinearity exponent $\rho > 1$. Let the Lebesgue index satisfy
		\begin{equation*}
			q \ge q_c := \frac{N(\rho-1)(1+\alpha_\infty)}{2}.
		\end{equation*}
		If $q > q_c$ (subcritical regime), there exists a maximal time $T_{\max} > 0$, depending exclusively on $\|u_0\|_{L^q}$, such that \eqref{eq:main} admits a unique mild solution $u \in E_{T_{\max}}$. If $q = q_c$ (critical regime), the same holds provided that $\|u_0\|_{L^{q_c}}$ is sufficiently small. Furthermore, in both cases, the data-to-solution map $u_0 \mapsto u$ is locally Lipschitz continuous from $L^q(\mathbb{R}^N)$ into $E_{T_{\max}}$, yielding local Hadamard well-posedness.
	\end{theorem}
	\begin{proof}
		The proof relies on the Banach Fixed Point Theorem applied to the integral operator
		\begin{equation*}
			\mathcal{T}(u)(t) = u_L(t) + \mathcal{N}(u)(t),
		\end{equation*}
		where $u_L(t) = S(t)u_0$ denotes the linear evolution and $\mathcal{N}(u)(t) = \int_0^t S(t-s)u(s)|u(s)|^{\rho-1} ds$ represents the nonlinear Volterra term. The analysis seeks a fixed point within the closed ball $B_R = \{ u \in E_T : \|u\|_{E_T} \le R \}$ for suitably chosen parameters $T > 0$ and $R > 0$.
		
		\noindent\textbf{Step 1: Index selection and integrability constraints.}
		Evaluating the nonlinear source in $L^r(\mathbb{R}^N)$ for $r = p/\rho$ yields $\|f(u)\|_{L^r} = \|u\|_{L^p}^\rho$. Applying the linear bound requires the temporal singularity exponents to satisfy
		\[ \gamma_{r,p} < 1, \quad \gamma_{r,q} < 1, \quad \text{and} \quad \rho \gamma_{q,p} < 1. \]
		Substituting $r = p/\rho$ into the decay rates produces
		\begin{equation*}
			\gamma_{r,p} = \frac{N(\rho-1)(1+\alpha_\infty)}{2p} = \frac{q_c}{p}.
		\end{equation*}
		Thus, $\gamma_{r,p} < 1$ requires $p > q_c$. The condition $\rho\gamma_{q,p} < 1$ imposes
		\begin{equation*}
			\rho\gamma_{q,p} = \rho \frac{q_c}{\rho-1} \left( \frac{1}{q} - \frac{1}{p} \right) < 1.
		\end{equation*}
		For the critical case $q = q_c$, this inequality demands $p < \rho q_c$, thereby fixing an arbitrary parameter $p \in (q_c, \rho q_c)$.
		
		For this choice of $p$, the condition $\gamma_{r,q} < 1$ holds, and the dimensional restriction \eqref{eq:dimensional_restriction} from Lemma \ref{lem:smoothing} is automatically satisfied since
		\begin{equation*}
			\frac{1}{r} - \frac{1}{p} = \frac{\rho-1}{p} < \frac{\rho-1}{q_c} = \frac{2}{N(1+\alpha_\infty)} < \frac{2}{N}.
		\end{equation*}
		This configuration ensures that both $1-\gamma_{r,p} > 0$ and $1-\rho\gamma_{q,p} > 0$, guaranteeing the convergence of the subsequent fractional integrals.
		
		\noindent\textbf{Step 2: Invariance of the resolution space ($\mathcal{T}(E_T) \subset E_T$).}
		For the linear evolution $u_L$, the strong continuity $u_L \in C([0,T]; L^q(\mathbb{R}^N))$ follows directly from the underlying resolvent family for $u_0 \in L^q(\mathbb{R}^N)$. The uniform weighted bound 
		\[ \sup_{0 < t \le T} t^{\gamma_{q,p}} \|S(t)u_0\|_{L^p} \le C_L \|u_0\|_{L^q} < \infty \] 
		is a direct consequence of Lemma \ref{lem:smoothing}. Verifying the vanishing condition at the origin employs a density argument. Let $\varepsilon > 0$. As $L^p(\mathbb{R}^N) \cap L^q(\mathbb{R}^N)$ is dense in $L^q(\mathbb{R}^N)$, there exists $\phi \in L^p(\mathbb{R}^N) \cap L^q(\mathbb{R}^N)$ such that $\|u_0 - \phi\|_{L^q} < \varepsilon$. By linearity and Lemma \ref{lem:smoothing}, it follows that
		\begin{equation*}
			t^{\gamma_{q,p}} \|S(t)u_0\|_{L^p} \le C_L \|u_0 - \phi\|_{L^q} + C t^{\gamma_{q,p}} \|\phi\|_{L^p} \le C_L \varepsilon + C t^{\gamma_{q,p}} \|\phi\|_{L^p}.
		\end{equation*}
		Given that the operator acts boundedly from $L^p$ to $L^p$, taking the limit superior as $t \to 0^+$ forces the second term to vanish, yielding
		\[ \limsup_{t \to 0^+} t^{\gamma_{q,p}} \|S(t)u_0\|_{L^p} \le C_L \varepsilon. \] 
		Since $\varepsilon > 0$ is arbitrary, the initial trace is verified, confirming $u_L \in E_T$.
		
		For the nonlinear component $\mathcal{N}(u)$, let $u \in E_T$. Applying the $L^r \to L^p$ linear smoothing estimate to the Volterra integral yields
		\begin{align*}
			\|\mathcal{N}(u)(t)\|_{L^p} &\le \int_0^t \|S(t-s) f(u(s))\|_{L^p} ds \\
			&\le C \int_0^t (t-s)^{-\gamma_{r,p}} \|f(u(s))\|_{L^r} ds\\
			&\le C \int_0^t (t-s)^{-\gamma_{r,p}} s^{-\rho\gamma_{q,p}} \left( s^{\gamma_{q,p}} \|u(s)\|_{L^p} \right)^\rho ds\\
			&\le C \int_0^t (t-s)^{-\gamma_{r,p}} s^{-\rho\gamma_{q,p}} ds \cdot \|u\|_{E_T}^\rho\\
			&= C t^{1 - \gamma_{r,p} - \rho\gamma_{q,p}} \int_0^1 (1-\tau)^{-\gamma_{r,p}} \tau^{-\rho\gamma_{q,p}} d\tau \cdot \|u\|_{E_T}^\rho.
		\end{align*}
		Multiplying by the temporal weight $t^{\gamma_{q,p}}$ consolidates the temporal exponent, providing the uniform bound
		\[ \sup_{0 < t \le T} t^{\gamma_{q,p}} \|\mathcal{N}(u)(t)\|_{L^p} \le C T^{1 - \gamma_{r,p} - (\rho-1)\gamma_{q,p}} \mathcal{B}(1-\gamma_{r,p}, 1-\rho\gamma_{q,p})\|u\|_{E_T}^\rho. \]
		Using the definition of the critical exponent $q_c$, it follows that $(\rho-1)\gamma_{q,p} = \frac{q_c}{q} - \frac{q_c}{p}$. Since $\gamma_{r,p} = \frac{q_c}{p}$, the full exponent simplifies structurally to
		\begin{equation*}
			1 - \gamma_{r,p} - (\rho-1)\gamma_{q,p} = 1 - \frac{q_c}{p} - \left( \frac{q_c}{q} - \frac{q_c}{p} \right) = 1 - \frac{q_c}{q}.
		\end{equation*}
		Substituting this relation recovers the uniform limit
		\begin{equation*}
			\sup_{0 < t \le T} t^{\gamma_{q,p}} \|\mathcal{N}(u)(t)\|_{L^p} \le C T^{1 - \frac{q_c}{q}} \mathcal{B}(1-\gamma_{r,p}, 1-\rho\gamma_{q,p}) \|u\|_{E_T}^\rho < \infty.
		\end{equation*}
		
		To verify the vanishing condition $\lim_{t \to 0^+} t^{\gamma_{q,p}} \|\mathcal{N}(u)(t)\|_{L^p} = 0$, in the subcritical regime ($q > q_c$), the decay $t^{1 - q_c/q} \to 0$ directly guarantees the limit. In the critical regime ($q = q_c$), where the temporal factor disappears ($1 - q_c/q = 0$), the non-decreasing continuous control function $K(t) = \sup_{0 < s \le t} s^{\gamma_{q,p}} \|u(s)\|_{L^p}$ is introduced. Since $u \in E_T$, it holds that $\lim_{t \to 0^+} K(t) = 0$. Utilizing the monotonicity $K(s) \le K(t)$ for $s \in (0,t]$, the integral estimate is refined to
		\begin{align*}
			t^{\gamma_{q,p}} \|\mathcal{N}(u)(t)\|_{L^p} &\le C t^{\gamma_{q,p}} \int_0^t (t-s)^{-\gamma_{r,p}} s^{-\rho\gamma_{q,p}} \left( s^{\gamma_{q,p}} \|u(s)\|_{L^p} \right)^\rho ds \\
			&\le C K(t)^\rho t^{\gamma_{q,p}} \int_0^t (t-s)^{-\gamma_{r,p}} s^{-\rho\gamma_{q,p}} ds\\
			&\le C \mathcal{B}(1-\gamma_{r,p}, 1-\rho\gamma_{q,p}) K(t)^\rho.
		\end{align*}
		As $t \to 0^+$, the right-hand side is governed by $K(t) \to 0$, thereby securing the required initial trace. Finally, the continuity $\mathcal{N}(u) \in C([0,T]; L^q(\mathbb{R}^N))$ follows from the strong continuity of the resolvent family $S(t)$, the integrability of the temporal kernel, and the Dominated Convergence Theorem. Thus, $\mathcal{T}$ maps $E_T$ into itself.
		
		\noindent\textbf{Step 3: Nonlinear estimates in $L^q$ and the contraction arguments.}
		To close the norm in $E_T$, the $L^q$ component is evaluated analogously. Employing the $L^r \to L^q$ linear bound yields
		\begin{align*}
			\|\mathcal{N}(u)(t)\|_{L^q} &\le C \int_0^t (t-s)^{-\gamma_{r,q}} \|u(s)\|_{L^p}^\rho ds \\
			&\le C \int_0^t (t-s)^{-\gamma_{r,q}} s^{-\rho\gamma_{q,p}} ds \cdot \|u\|_{E_T}^\rho \\
			&\le C t^{1 - \gamma_{r,q} - \rho\gamma_{q,p}} \mathcal{B}(1-\gamma_{r,q}, 1-\rho\gamma_{q,p}) \|u\|_{E_T}^\rho.
		\end{align*}
		A direct computation confirms $1 - \gamma_{r,q} - \rho\gamma_{q,p} = 1 - q_c/q$. Taking the supremum over $(0,T]$ produces
		\begin{equation*}
			\sup_{0 \le t \le T} \|\mathcal{N}(u)(t)\|_{L^q} \le C_q T^{1 - \frac{q_c}{q}} \|u\|_{E_T}^\rho.
		\end{equation*}
		Combining both components provides the bound $\|\mathcal{N}(u)\|_{E_T} \le C_N T^{1 - \frac{q_c}{q}} \|u\|_{E_T}^\rho$.
		
		The corresponding difference satisfies the Lipschitz estimate
		\begin{equation*}
			\|\mathcal{N}(u) - \mathcal{N}(v)\|_{E_T} \le 2\rho C_N T^{1 - \frac{q_c}{q}} R^{\rho-1} \|u - v\|_{E_T}.
		\end{equation*}
		For the linear evolution, let $\|u_L\|_{E_T} \le 2C_L \|u_0\|_{L^q} := \delta$. The operator $\mathcal{T}$ acts as a contraction on $B_R$ provided that
		\begin{equation*}
			\delta + C_N T^{1 - \frac{q_c}{q}} R^\rho \le R, \quad \text{and} \quad 2\rho C_N T^{1 - \frac{q_c}{q}} R^{\rho-1} \le \frac{1}{2}.
		\end{equation*}
		
		\noindent\textit{Case I: Subcritical regime ($q > q_c$).}
		Here, $1 - q_c/q > 0$. By fixing $R = 2\delta$, for any initial datum size $\|u_0\|_{L^q}$, there exists a sufficiently small $T_{\max} > 0$ satisfying both conditions, ensuring unconditional local existence.
		
		\noindent\textit{Case II: Critical regime ($q = q_c$).}
		Here, $1 - q_c/q = 0$. To enforce the contraction property independently of $T$, the radius is explicitly constrained to $R \le (4\rho C_N)^{-\frac{1}{\rho-1}}$. The first condition subsequently requires $\delta \le R/2$, which imposes the smallness bound
		\begin{equation*}
			\|u_0\|_{L^{q_c}} \le \frac{1}{4C_L} (4\rho C_N)^{-\frac{1}{\rho-1}}.
		\end{equation*}
		Under this specific constraint, $\mathcal{T}$ forms a strict contraction on $B_R$ for any $T > 0$.
		
		\noindent\textbf{Step 4: Uniqueness in the resolution space $E_T$.}
		To secure local Hadamard well-posedness within this functional framework, uniqueness must be established across the entire resolution space $E_T$. Let $u, v \in E_T$ be two mild solutions to \eqref{eq:main} arising from the same initial datum $u_0$. The non-decreasing control function for their difference is defined as
		\[M(t) = \sup_{0 < s \le t} s^{\gamma_{q,p}} \|u(s) - v(s)\|_{L^p},\]
		alongside the individual controls 
		\[K_w(t) = \sup_{0 < s \le t} s^{\gamma_{q,p}} \|w(s)\|_{L^p},\] 
		for $w \in \{u, v\}$. Given that $u, v \in E_T$, the vanishing condition implies $\lim_{t \to 0^+} K_u(t) = \lim_{t \to 0^+} K_v(t) = 0$.
		
		Evaluating the difference between their respective integral formulations yields
		\begin{equation*}
			u(t) - v(t) = \int_0^t S(t-s) [f(u(s)) - f(v(s))] ds.
		\end{equation*}
		Applying the $L^r \to L^p$ linear smoothing estimate produces
		\begin{align*}
			\|u(t) - v(t)\|_{L^p} &\le C \int_0^t (t-s)^{-\gamma_{r,p}} \|f(u(s)) - f(v(s))\|_{L^r} ds \\
			&\le C \int_0^t (t-s)^{-\gamma_{r,p}} \left( \|u(s)\|_{L^p}^{\rho-1} + \|v(s)\|_{L^p}^{\rho-1} \right) \|u(s) - v(s)\|_{L^p} ds.
		\end{align*}
		Multiplying by the temporal weight $t^{\gamma_{q,p}}$ and introducing the fractional partition $s^{-\rho\gamma_{q,p}} s^{\rho\gamma_{q,p}}$ within the integrand to reconstruct the weighted norm topologies yields
		\begin{align*}
			t^{\gamma_{q,p}} \|u(t) - v(t)\|_{L^p} &\le C t^{\gamma_{q,p}} \int_0^t (t-s)^{-\gamma_{r,p}} s^{-\rho\gamma_{q,p}} \left[ s^{\gamma_{q,p}(\rho-1)}\left( \|u(s)\|_{L^p}^{\rho-1} + \|v(s)\|_{L^p}^{\rho-1} \right) \right] \\
			&\quad \times \left( s^{\gamma_{q,p}} \|u(s) - v(s)\|_{L^p} \right) ds.
		\end{align*}
		Evaluating the temporal integral via the Beta function and taking the supremum over $(0, T^\ast]$ for some $T^\ast \le T$ ensures
		\begin{equation*}
			M(T^\ast) \le C \mathcal{B}(1-\gamma_{r,p}, 1-\rho\gamma_{q,p}) (T^\ast)^{1-\frac{q_c}{q}} \left( K_u(T^\ast)^{\rho-1} + K_v(T^\ast)^{\rho-1} \right) M(T^\ast).
		\end{equation*}
		As $\lim_{t \to 0^+} K_u(t) = \lim_{t \to 0^+} K_v(t) = 0$, the temporal coefficient on the right-hand side vanishes as $T^\ast \to 0^+$. This limit holds independently across both the subcritical and critical regimes. Hence, $T^\ast > 0$ can be chosen sufficiently small so that the prefactor is bounded by $1/2$, giving $M(T^\ast) \le \frac{1}{2} M(T^\ast)$. This implies $M(T^\ast) = 0$, proving $u \equiv v$ on $[0, T^\ast]$.
		
		To extend this local uniqueness to the maximal existence interval, let 
		\[T_\ast = \sup \{ \tau \in (0, T] : u \equiv v \text{ on } [0, \tau] \}.\]
		By definition, $T_\ast \ge T^\ast > 0$. Assuming, by contradiction, that $T_\ast < T$, the integral of the difference over the history segment $[0, T_\ast]$ vanishes identically. For $t > T_\ast$, the difference satisfies
		\begin{equation*}
			u(t) - v(t) = \int_{T_\ast}^t S(t-s) [f(u(s)) - f(v(s))] ds.
		\end{equation*}
		For $s \in [T_\ast, T]$, the temporal singularity at the origin is bypassed, and the continuous solutions remain bounded in $L^p(\mathbb{R}^N)$. Let $L = \sup_{s \in [T_\ast, T]} (\|u(s)\|_{L^p}^{\rho-1} + \|v(s)\|_{L^p}^{\rho-1}) < \infty$. Applying the linear smoothing estimates directly provides
		\begin{equation*}
			\|u(t) - v(t)\|_{L^p} \le C L \int_{T_\ast}^t (t-s)^{-\gamma_{r,p}} \|u(s) - v(s)\|_{L^p} ds.
		\end{equation*}
		Since $\gamma_{r,p} = q_c/p < 1$, the convolution kernel remains weakly singular. Invoking the generalized Gr\"onwall inequality (cf. Henry \cite{Henry1981}), it follows that $\|u(t) - v(t)\|_{L^p} = 0$ on an extended interval $[T_\ast, T_\ast + \delta]$ for some $\delta > 0$, contradicting the maximality of $T_\ast$. Therefore, $T_\ast = T$, establishing global uniqueness in $E_T$.
		
		\noindent\textbf{Step 5: Continuous dependence on initial data.}
		To complete the proof, the difference between two mild solutions $u$ and $v$ arising from initial data $u_0$ and $v_0$ is evaluated within the contraction ball $B_R$. Utilizing the Lipschitz constant bounded by $1/2$ from Step 3, the difference is estimated as
		\begin{align*}
			\|u - v\|_{E_T} &= \|\mathcal{T}(u) - \mathcal{T}(v)\|_{E_T} \\
			&\le \|S(t)(u_0 - v_0)\|_{E_T} + \|\mathcal{N}(u) - \mathcal{N}(v)\|_{E_T} \\
			&\le 2C_L \|u_0 - v_0\|_{L^q} + \frac{1}{2} \|u - v\|_{E_T}.
		\end{align*}
		This algebraically yields $\|u - v\|_{E_T} \le 4C_L \|u_0 - v_0\|_{L^q}$, establishing Lipschitz continuous dependence on the initial data within the weighted topology. This confirms local Hadamard well-posedness in the resolution space $E_T$ and concludes the proof.
	\end{proof}
	
	\subsection{Supercritical ill-posedness via norm inflation}
	To justify defining $q_c$ as the critical threshold, the Cauchy problem is shown to be ill-posed in the supercritical regime $1 < q < q_c$. Specifically, the data-to-solution map exhibits instantaneous norm inflation in the sense introduced by Christ, Colliander, and Tao \cite{Christ2003}. This mechanism is driven by the scaling mismatch between the memory kernel's asymptotic regularization and the nonlinear amplification.
	
	\begin{theorem}\label{thm:ill_posedness}
		Let $g$ satisfy (H1)-(H2) and let $1 < q < q_c$. The Cauchy problem \eqref{eq:main} is ill-posed at the origin in the following sense: for any $\epsilon > 0$, there exists a sequence of localized initial data $u_{0,k} \in L^q(\mathbb{R}^N)$ and a sequence of positive times $t_k \to 0^+$ such that 
		\begin{equation*}
			\|u_{0,k}\|_{L^q(\mathbb{R}^N)} < \epsilon, \quad \forall k \in \mathbb{N},
		\end{equation*}
		whereas the corresponding first Picard iterates satisfy the instantaneous norm inflation
		\begin{equation*}
			\lim_{k \to \infty} \|\mathcal{N}_1(u_{0,k})(t_k)\|_{L^q(\mathbb{R}^N)} = \infty.
		\end{equation*}
		Consequently, the data-to-solution map $u_0 \mapsto u$ fails to be uniformly continuous at the origin in the $L^q(\mathbb{R}^N)$ topology.
	\end{theorem}
	\begin{proof}
		Assuming, for the sake of contradiction, that the Cauchy problem is locally well-posed in $L^q(\mathbb{R}^N)$ in the Hadamard sense, the data-to-solution map must be uniformly continuous near the origin. This continuity implies that its first Picard iterate,
		\begin{equation*}
			\mathcal{N}_1(u_0)(t) = \int_0^t S(t-s) f(S(s)u_0) ds,
		\end{equation*}
		must satisfy the uniform local estimate
		\begin{equation*}
			\|\mathcal{N}_1(u_0)(t)\|_{L^q(\mathbb{R}^N)} \le C \|u_0\|_{L^q(\mathbb{R}^N)}^\rho,
		\end{equation*}
		for all sufficiently small data.
		
		Let $\phi \in \mathcal{S}(\mathbb{R}^N)$ be a non-negative cut-off profile such that $\phi \not\equiv 0$. For a scaling parameter $\lambda \ge 1$ and an amplitude $A_\lambda > 0$, the localized initial datum is defined as $u_{0,\lambda}(x) = A_\lambda \phi(\lambda x)$, which scales as $\|u_{0,\lambda}\|_{L^q(\mathbb{R}^N)} = A_\lambda \lambda^{-N/q} \|\phi\|_{L^q(\mathbb{R}^N)}$. Let $\sigma = \frac{2}{1+\alpha_\infty}$. Setting $y = \lambda x$ and $\zeta = \xi/\lambda$, the linear evolution driven by the resolvent family condenses into
		\begin{equation*}
			S(\theta \lambda^{-\sigma})u_{0,\lambda}(y/\lambda) = A_\lambda \big[ K(\theta, \lambda)\phi \big](y),
		\end{equation*}
		where $K(\theta, \lambda)$ is defined via its Fourier multiplier
		\begin{equation*}
			K(\theta, \lambda)\phi(y) = \frac{1}{(2\pi)^N} \int_{\mathbb{R}^N} e^{iy \cdot \zeta} \widetilde{\Psi}(\theta \lambda^{-\sigma}, \zeta \theta^{1/\sigma}) \hat{\phi}(\zeta) d\zeta.
		\end{equation*}
		
		Substituting $f(S(\theta \lambda^{-\sigma})u_{0,\lambda}) = A_\lambda^\rho f(K(\theta, \lambda)\phi)$ into the integral formulation evaluated at the trajectory time $t = \tau \lambda^{-\sigma}$, the expression becomes
		\begin{equation*}
			\mathcal{N}_1(u_{0,\lambda})(\tau \lambda^{-\sigma}, y/\lambda) = \lambda^{-\sigma} A_\lambda^\rho \int_0^\tau K(\tau-\theta, \lambda) \big[ f(K(\theta, \lambda)\phi) \big](y) d\theta.
		\end{equation*}
		Taking the $L^q$-norm with respect to $x = y/\lambda$ extracts the volumetric factor $\lambda^{-N/q}$, yielding
		\begin{equation}\label{eq:picard_lq_prelimit}
			\|\mathcal{N}_1(u_{0,\lambda})(\tau \lambda^{-\sigma}, \cdot)\|_{L^q(\mathbb{R}^N)} = A_\lambda^\rho \lambda^{-\sigma - \frac{N}{q}} \| F_\lambda \|_{L^q(\mathbb{R}^N)},
		\end{equation}
		where $F_\lambda(y) := \int_0^\tau K(\tau-\theta, \lambda) \big[ f(K(\theta, \lambda)\phi) \big](y) d\theta$.
		
		To obtain a uniform lower bound for the pre-limit profile, the spatial mass of $F_\lambda$ in the frequency domain at the origin is evaluated as 
		\begin{equation*}
			\widehat{F_\lambda}(0) = \int_0^\tau \widehat{K}(\tau-\theta, \lambda; 0) \cdot \mathcal{F}\big\{ f(K(\theta, \lambda)\phi) \big\}(0) d\theta.
		\end{equation*}
		Since $\widehat{K}(t, \lambda; 0) = \widetilde{\Psi}(t \lambda^{-\sigma}, 0) = 1$, the mass simplifies to
		\begin{equation*}
			\widehat{F_\lambda}(0) = \int_0^\tau \left( \int_{\mathbb{R}^N} f(K(\theta, \lambda)\phi(y)) dy \right) d\theta.
		\end{equation*}
		
		Next, the nonlinear mass functional $\mathcal{I}_\lambda(\theta) := \int_{\mathbb{R}^N} f(K(\theta, \lambda)\phi(y)) dy$ is shown to remain uniformly close to its initial state for all $\lambda \ge 1$. At $\theta = 0$, $\mathcal{I}_\lambda(0) = \|\phi\|_{L^\rho(\mathbb{R}^N)}^\rho$. For $\theta > 0$, the multiplier error is analyzed. Setting $t = \theta \lambda^{-\sigma}$, the difference symbol $\widetilde{\Psi}(t, \eta) - 1$ with $\eta = \zeta \theta^{1/\sigma}$ is represented via the Bromwich contour $\Gamma'$ as
		\begin{equation*}
			\widetilde{\Psi}(t, \eta) - 1 = -|\eta|^2 \frac{1}{2\pi i} \int_{\Gamma'} \frac{e^\mu m(t,\mu)}{\mu(\mu + |\eta|^2 m(t,\mu))} d\mu,
		\end{equation*}
		where $m(t,\mu) = \mu^{-\alpha_\infty} (\mu/t)^{\alpha_\infty} \hat{g}(\mu/t)$. To establish the uniform regularity of this transition, the auxiliary symbol is defined as
		\begin{equation*}
			m_\lambda(\zeta) := \frac{\widetilde{\Psi}(\theta \lambda^{-\sigma}, \zeta \theta^{1/\sigma}) - 1}{|\zeta \theta^{1/\sigma}|^2} = -\frac{1}{2\pi i} \int_{\Gamma'} \frac{e^\mu m(t,\mu)}{\mu(\mu + |\zeta \theta^{1/\sigma}|^2 m(t,\mu))} d\mu.
		\end{equation*}
		To ensure that $m_\lambda(\zeta)$ belongs uniformly to the H\"ormander class $S^0_{1,0}$ without scaling degeneration as $\lambda \to \infty$ (and thus $t \to 0$), the limit $\zeta \to 0$ is explicitly controlled. The derivative with respect to the frequency squared yields the principal value at the origin:
		\begin{equation*}
			M_0(0) := \lim_{t \to 0^+} m_\lambda(0) = -\frac{1}{2\pi i} \int_{\Gamma'} \frac{e^\mu m(0,\mu)}{\mu^2} d\mu.
		\end{equation*}
		By hypothesis (H2), as $\lambda \to \infty$ and $t \to 0$, the high-frequency limit stabilizes to $m(0,\mu) = C_1 \mu^{-\alpha_\infty}$. Thus, the limiting integral evaluates precisely to the canonical Hankel representation of the reciprocal Gamma function
		\begin{equation*}
			M_0(0) = -\frac{C_1}{2\pi i} \int_{\Gamma'} \frac{e^\mu}{\mu^{2+\alpha_\infty}} d\mu = -\frac{C_1}{\Gamma(2+\alpha_\infty)}.
		\end{equation*}
		As $C_1 > 0$ and the Gamma function is positive for positive real arguments, the residue constitutes a non-zero finite constant, confirming the non-degeneracy of the multiplier at the origin. 
		
		To establish uniform bounds for the spatial derivatives, $\nabla_\zeta m_\lambda(\zeta)$ is evaluated. Applying the chain rule under the integral sign produces an inner factor of $2\zeta \theta^{2/\sigma}$, which is balanced by the squared parameter $|\zeta \theta^{1/\sigma}|^2$ within the denominator. Consequently, the uniform coercivity property $|D| \ge C(r + |\zeta \theta^{1/\sigma}|^2 r^{-\alpha_\infty})$ ensures stable bounds, yielding
		\begin{equation*}
			|\nabla_\zeta^\beta m_\lambda(\zeta)| \le C_\beta (1+|\zeta \theta^{1/\sigma}|)^{-|\beta|} \le C_\beta (1+|\zeta|)^{-|\beta|},
		\end{equation*}
		uniformly for all $\lambda \ge 1$ and $\theta \le 1$. By the Mikhlin-H\"ormander Multiplier Theorem, the operator associated with $m_\lambda$ is bounded from $L^\rho(\mathbb{R}^N)$ to $L^\rho(\mathbb{R}^N)$ with a constant $C_M$ independent of $\lambda$. Since $K(\theta, \lambda)\phi - \phi = \theta^{2/\sigma} \mathcal{F}^{-1} \big( m_\lambda(\zeta) \mathcal{F}(-\Delta \phi) \big)$ and $-\Delta \phi \in L^\rho(\mathbb{R}^N)$, it follows that
		\begin{equation*}
			\|K(\theta, \lambda)\phi - \phi\|_{L^\rho(\mathbb{R}^N)} \le C_M \theta^{2/\sigma} \|\Delta \phi\|_{L^\rho(\mathbb{R}^N)} = C_\phi \theta^{2/\sigma},
		\end{equation*}
		where $C_\phi$ is invariant with respect to $\lambda$. 
		
		Bounding the mass deviation relies on the elementary inequality $|f(a)-f(b)| \le C_\rho \big(|a|^{\rho-1} + |b|^{\rho-1}\big)|a-b|$. Integrating over $\mathbb{R}^N$ and applying H\"older's inequality with conjugate exponents $\rho$ and $\frac{\rho}{\rho-1}$ yields
		\begin{align*}
			|\mathcal{I}_\lambda(\theta) - \mathcal{I}_\lambda(0)| &\le \int_{\mathbb{R}^N} |f(K(\theta, \lambda)\phi) - f(\phi)| dy \\
			&\le C_\rho \left( \|K(\theta, \lambda)\phi\|_{L^\rho}^{\rho-1} + \|\phi\|_{L^\rho}^{\rho-1} \right) \|K(\theta, \lambda)\phi - \phi\|_{L^\rho}.
		\end{align*}
		Given that linear convolution operators acting on the Schwartz class preserve integrability, the multiplier terms in the parentheses are bounded by a universal constant depending only on $\phi$. Substituting the uniform error bound produces
		\begin{equation*}
			|\mathcal{I}_\lambda(\theta) - \mathcal{I}_\lambda(0)| \le \widetilde{C}_\phi \theta^{2/\sigma}.
		\end{equation*}
		Choosing $\tau^\ast = \left( \frac{\|\phi\|_{L^\rho}^\rho}{2 \widetilde{C}_\phi} \right)^{\sigma/2} > 0$ guarantees $\mathcal{I}_\lambda(\theta) \ge \frac{1}{2}\|\phi\|_{L^\rho(\mathbb{R}^N)}^\rho$, which implies $\widehat{F_\lambda}(0) \ge c_0 > 0$ uniformly for all $\lambda \ge 1$.
		
		To address the spatial decay of $F_\lambda(y)$ near the singular limit $\theta \to 0$, the spatial kernel $\mathcal{G}_t(x) = \mathcal{F}^{-1}(\widetilde{\Psi}(t,\cdot))(x)$ associated with the H\"ormander class $S^{-2}_{1,0}$ is analyzed under the scaling dilation $\theta^{-N/\sigma}\mathcal{G}_t(y\theta^{-1/\sigma})$. To prevent the potential breakdown of spatial localization as $\theta \to 0$, the uniform envelope function is introduced:
		\begin{equation*}
			\mathcal{E}(y) := \sup_{\theta \in (0, \tau^\ast]} \left| \left[ \theta^{-N/\sigma} \mathcal{G}_t(\cdot \theta^{-1/\sigma}) \ast f(K(\theta, \lambda)\phi) \right](y) \right|.
		\end{equation*}
		Applying Peetre's inequality (see Appendix A for the detailed scaling neutralization and uniform bounds), the convolution integral can be decomposed. Since $\phi \in \mathcal{S}(\mathbb{R}^N)$, the power-type nonlinearity satisfies the faster decay profile $|f(K(\theta,\lambda)\phi(z))| \le C (1+|z|)^{-\rho(N+2)}$. Because $\rho > 1$, the excess polynomial weight is absorbed, yielding
		\begin{align*}
			|F_\lambda(y)| &\le \int_0^{\tau^\ast} \int_{\mathbb{R}^N} \theta^{-N/\sigma} |\mathcal{G}_t((y-z)\theta^{-1/\sigma})| |f(K(\theta, \lambda)\phi(z))| dz d\theta \\
			&\le C_0 (1+|y|)^{-(N+2)},
		\end{align*}
		where the constant $C_0 > 0$ is invariant during the transient limit $\theta \to 0$, as the regularizing properties of $\phi$ prevent the concentration of the Dirac mass from inducing norm explosions. 
		
		With this uniform polynomial envelope established, the strictly positive spatial mass $\widehat{F_\lambda}(0) \ge c_0 > 0$ connects directly to the required $L^q$ lower bound. Splitting the spatial integration over a ball $B_R$ and its complement yields
		\begin{equation*}
			c_0 \le \int_{\mathbb{R}^N} F_\lambda(y) dy \le \int_{|y| \le R} |F_\lambda(y)| dy + \int_{|y| > R} |F_\lambda(y)| dy.
		\end{equation*}
		Applying the uniform envelope $|F_\lambda(y)| \le C_0(1+|y|)^{-(N+2)}$, the tail integral is heavily suppressed for $N \ge 1$:
		\begin{equation*}
			\int_{|y| > R} |F_\lambda(y)| dy \le C_0 \int_{|y| > R} |y|^{-(N+2)} dy \le \frac{C_0'}{R^2}.
		\end{equation*}
		Choosing $R > 0$ sufficiently large such that the tail is bounded by $c_0/2$ confines the essential mass to the compact domain $B_R$, guaranteeing 
		\[\int_{|y| \le R} |F_\lambda(y)| dy \ge c_0/2.\] 
		Applying H\"older's inequality on this bounded support extracts the uniform Lebesgue lower bound
		\begin{equation*}
			\frac{c_0}{2} \le |B_R|^{1 - \frac{1}{q}} \|F_\lambda\|_{L^q(B_R)} \implies \|F_\lambda\|_{L^q(\mathbb{R}^N)} \ge \|F_\lambda\|_{L^q(B_R)} \ge c_1 > 0,
		\end{equation*}
		where $c_1 := \frac{c_0}{2} |B_R|^{\frac{1}{q}-1}$ is independent of the modulation parameter $\lambda$. 
		
		Substituting this into \eqref{eq:picard_lq_prelimit} and choosing the supercritical amplitude $A_\lambda = \epsilon \lambda^{N/q}(\ln \lambda)^{-1}$ yields the lower bound
		\begin{equation*}
			\|\mathcal{N}_1(u_{0,\lambda})(\tau^\ast \lambda^{-\sigma}, \cdot)\|_{L^q(\mathbb{R}^N)} \ge c_1 \epsilon^\rho (\ln \lambda)^{-\rho} \lambda^{(\rho-1)\frac{N}{q} - \sigma}.
		\end{equation*}
		In the supercritical regime $1 < q < q_c$, replacing $\sigma = \frac{2}{1+\alpha_\infty}$ and utilizing the definition of the critical threshold $q_c$ reveals that the structural exponent satisfies
		\begin{equation*}
			\delta := (\rho-1)\frac{N}{q} - \sigma = \frac{2(q_c - q)}{q(1+\alpha_\infty)}.
		\end{equation*}
		Considering that $q < q_c$, the exponent $\delta$ is strictly positive. Extracting a discrete sequence $\lambda_k \to \infty$ and evaluating at the corresponding  times $t_k := \tau^\ast \lambda_k^{-\sigma} \to 0^+$, the initial data is defined as $u_{0,k} := u_{0,\lambda_k}$. While the initial profile satisfies $\|u_{0,k}\|_{L^q} \le \epsilon (\ln \lambda_k)^{-1} \to 0$, the Picard iterate collapses as
		\begin{equation*}
			\|\mathcal{N}_1(u_{0,k})(t_k, \cdot)\|_{L^q(\mathbb{R}^N)} \ge c_1 \epsilon^\rho (\ln \lambda_k)^{-\rho} \lambda_k^{\delta} \xrightarrow{k \to \infty} \infty.
		\end{equation*}
		This divergence confirms that the data-to-solution map fails to be uniformly continuous at the origin along the sequence $t_k$, completing the proof.
	\end{proof}

\section{Global existence and asymptotic decay}
In contrast to local well-posedness, which is governed by the short-time super-diffusive regularization $\alpha_\infty$, the global dynamics depend on the long-time memory tail $\alpha_0$. Establishing global existence requires tracking the solution across the transition between these two frequency regimes. To define the resolution space, the temporal singularity exponents from Lemma \ref{lem:smoothing} and Lemma \ref{lem:smoothing_asymptotic} are recalled. For $1 \le a \le p < \infty$, the high-frequency and low-frequency smoothing indices are respectively denoted by
\begin{equation*}
	\gamma_{a,p} = \frac{N(1+\alpha_\infty)}{2}\left(\frac{1}{a} - \frac{1}{p}\right), \quad \text{and} \quad \beta_{a,p} = \frac{N(1+\alpha_0)}{2}\left(\frac{1}{a} - \frac{1}{p}\right).
\end{equation*}
The asymptotic decay demands initial data in a lower Lebesgue space $L^m(\mathbb{R}^N)$, where $m < q_c$. The integrability of the Volterra integral at infinity reveals the critical mass threshold $m_c$. For the solution to decay globally without experiencing finite-time blow-up, operating in the regime $m_c > 1$ is structurally necessary. This barrier recovers the nonlocal Fujita-type critical exponent, denoted by $\rho_F$:
\begin{equation*}
	m_c := \frac{N(\rho-1)(1+\alpha_0)}{2}, \quad \text{and} \quad \rho_F := 1 + \frac{2}{N(1+\alpha_0)}.
\end{equation*}
Furthermore, to unify the short-time and long-time integrations, the nonlinearity must bridge the spectral gap between the regimes, requiring $\rho > \frac{1+\alpha_\infty}{1+\alpha_0}$. This motivates the definition of the critical threshold for global existence:
\begin{equation*}
	\rho_0 := \max\left( \rho_F, \frac{1+\alpha_\infty}{1+\alpha_0} \right).
\end{equation*}
Let $p = \rho m$. To reflect the dual temporal scaling, the global resolution space $X_\infty$ is constructed through the continuous weight function
\begin{equation*}
	\Theta(t) = \max \left( t^{\gamma_{q_c, p}}, t^{\beta_{m, p}} \right).
\end{equation*}
Consequently, the Banach space $X_\infty = \{ u \in C((0,\infty); L^p(\mathbb{R}^N)) : \|u\|_{X_\infty} < \infty \}$ is formalized, endowed with the weighted norm
\begin{equation*}
	\|u\|_{X_\infty} = \sup_{t > 0} \Theta(t) \|u(t)\|_{L^p(\mathbb{R}^N)}.
\end{equation*}
This functional framework, relying on piecewise temporal weights to absorb the local singularity while capturing the asymptotic behavior, is inspired by the contraction arguments of Weissler \cite{Weissler1981}.

\begin{theorem}\label{thm:global_existence}
	Let $g$ satisfy (H1), (H2) and (H3) and assume the nonlinearity satisfies $\rho > \rho_0$. Let the initial datum $u_0 \in L^{q_c}(\mathbb{R}^N) \cap L^m(\mathbb{R}^N)$ for some fixed $m \in \left(\frac{q_c}{\rho}, m_c\right)$. There exists $\delta > 0$ such that if $\|u_0\|_{L^{q_c}} + \|u_0\|_{L^m} \le \delta$, the Cauchy problem \eqref{eq:main} admits a unique global-in-time mild solution $u \in X_\infty$.
\end{theorem}
\begin{proof}
	Mild solutions are constructed as fixed points of the operator 
	\[\mathcal{T}(u)(t) = u_L(t) + \mathcal{N}(u)(t),\] 
	where $u_L(t) = S(t)u_0$ and $\mathcal{N}(u)(t) = \int_0^t S(t-s)f(u(s)) ds$, within a closed ball $B_R \subset X_\infty$. Selecting $m > \frac{q_c}{\rho}$ guarantees $p > q_c$, which ensures the topological bounds required for the short-time regime.
	
	\noindent\textbf{Step 1: Unified linear estimates.}
	Combining the estimates from Lemma \ref{lem:smoothing} and Lemma \ref{lem:smoothing_asymptotic}, the resolvent operator satisfies both temporal decay rates for all $t > 0$. Considering that $u_0 \in L^{q_c}(\mathbb{R}^N) \cap L^m(\mathbb{R}^N)$, the linear evolution is bounded directly by
	\begin{align*}
		\|u_L(t)\|_{L^p} &\le C \min \left( t^{-\gamma_{q_c,p}} \|u_0\|_{L^{q_c}}, t^{-\beta_{m,p}} \|u_0\|_{L^m} \right) \\
		&\le C \min \left( t^{-\gamma_{q_c,p}}, t^{-\beta_{m,p}} \right) (\|u_0\|_{L^{q_c}} + \|u_0\|_{L^m}).
	\end{align*}
	Using the identity $\min(t^{-\gamma_{q_c,p}}, t^{-\beta_{m,p}}) \equiv \Theta(t)^{-1}$, it follows that
	\begin{equation*}
		\|u_L(t)\|_{L^p} \le C \Theta(t)^{-1} (\|u_0\|_{L^{q_c}} + \|u_0\|_{L^m}).
	\end{equation*}
	Multiplying by $\Theta(t)$ and taking the supremum over $(0, \infty)$ yields
	\begin{equation*}
		\|u_L\|_{X_\infty} \le C_L (\|u_0\|_{L^{q_c}} + \|u_0\|_{L^m}) =: \delta_0.
	\end{equation*}
	The continuity $u_L \in C((0,\infty); L^p(\mathbb{R}^N))$ follows from the strong continuity of $S(t)$.
	
	\noindent\textbf{Step 2: Nonlinear estimates in the asymptotic regime.}
	For any $u \in X_\infty$, it follows by definition that $\|u(s)\|_{L^p} \le \Theta(s)^{-1} \|u\|_{X_\infty}$ for all $s > 0$. For bounded time intervals, say $t \le 2$, the local integrability of the Volterra kernel is governed by the short-time singularity $\gamma_{q_c,p}$. Since the inequality $\Theta(s)^{-1} = \min(s^{-\gamma_{q_c, p}}, s^{-\beta_{m, p}}) \le s^{-\gamma_{q_c, p}}$ holds for all $s > 0$, the temporal singularity on bounded intervals behaves as in Theorem \ref{thm:local_existence}, yielding
	\begin{equation*}
		\|\mathcal{N}(u)(t)\|_{L^p} \le C t^{-\gamma_{q_c,p}} \|u\|_{X_\infty}^\rho.
	\end{equation*}
	
	For $t > 2$, the Volterra integral is split at the midpoint to isolate the memory effects:
	\begin{equation*}
		\mathcal{N}(u)(t) = \int_0^{t/2} S(t-s)f(u(s)) ds + \int_{t/2}^t S(t-s)f(u(s)) ds := I_1(t) + I_2(t).
	\end{equation*}
	For $I_1(t)$, the elapsed time satisfies $t-s \ge t/2 > 1$. Applying Lemma \ref{lem:smoothing_asymptotic} with base space $L^m$, and noting that $\|f(u(s))\|_{L^m} = \|u(s)\|_{L^p}^\rho$, the estimate becomes
	\begin{align*}
		\|I_1(t)\|_{L^p} &\le C \int_0^{t/2} (t-s)^{-\beta_{m,p}} \|u(s)\|_{L^p}^\rho ds \\
		&\le C t^{-\beta_{m,p}} \left( \int_0^1 s^{-\rho \gamma_{q_c,p}} ds + \int_1^{t/2} s^{-\rho \beta_{m,p}} ds \right) \|u\|_{X_\infty}^\rho.
	\end{align*}
	The first integral is finite because $q_c$ forces $\rho \gamma_{q_c,p} < 1$. For the second integral, the structural condition $m < m_c$ ensures that
	\begin{equation*}
		\rho \beta_{m,p} = \frac{N(1+\alpha_0)}{2} \left( \frac{\rho}{m} - \frac{\rho}{p} \right) = \frac{N(\rho-1)(1+\alpha_0)}{2m} = \frac{m_c}{m} > 1.
	\end{equation*}
	Consequently, the integral 
	\begin{equation*}
		\int_1^\infty s^{-\rho \beta_{m,p}} ds
	\end{equation*}
	converges, producing the bound 
	\begin{equation*}
		\|I_1(t)\|_{L^p} \le C t^{-\beta_{m,p}} \|u\|_{X_\infty}^\rho.
	\end{equation*}
	
	For $I_2(t)$, the evaluation occurs at $s \ge t/2 > 1$, which implies 
	\begin{equation*}
		\|u(s)\|_{L^p}^\rho \le C t^{-\rho \beta_{m,p}} \|u\|_{X_\infty}^\rho.
	\end{equation*}
	Setting $\tau = t-s \in [0, t/2]$ and applying the unified $L^m \to L^p$ bound yields
	\begin{align*}
		\|I_2(t)\|_{L^p} &\le C t^{-\rho \beta_{m,p}} \|u\|_{X_\infty}^\rho \int_0^{t/2} \min(\tau^{-\gamma_{m,p}}, \tau^{-\beta_{m,p}}) d\tau \\
		&\le C t^{-\rho \beta_{m,p}} (1 + t^{1-\beta_{m,p}}) \|u\|_{X_\infty}^\rho.
	\end{align*}
	Considering that $\gamma_{m,p} = \frac{q_c}{\rho m} < 1$, the integral is locally finite, and since $\rho \beta_{m,p} > 1$, the global integrability is preserved. Expanding the algebraic product yields
	\begin{equation*}
		\|I_2(t)\|_{L^p} \le C \left( t^{-\rho \beta_{m,p}} + t^{1 - \beta_{m,p}(\rho + 1)} \right) \|u\|_{X_\infty}^\rho.
	\end{equation*}
	Both temporal exponents dominate the target asymptotic decay $t^{-\beta_{m,p}}$. Specifically, since $\rho > 1$ and $\beta_{m,p} > 0$, it holds that $-\rho \beta_{m,p} < -\beta_{m,p}$. Furthermore, the critical integrability condition $\rho \beta_{m,p} > 1$ guarantees that
	\begin{equation*}
		1 - \beta_{m,p}(\rho + 1) = 1 - \rho \beta_{m,p} - \beta_{m,p} < -\beta_{m,p}.
	\end{equation*}
	Consequently, $I_2(t)$ decays faster than the principal memory tail governed by $I_1(t)$, ensuring it is structurally absorbed into the target bound. Multiplying by the dual weight $\Theta(t)$ and taking the supremum over $(0,\infty)$ yields 
	\begin{equation*}
		\sup_{t>0} \Theta(t)\|\mathcal{N}(u)(t)\|_{L^p} \le C_N \|u\|_{X_\infty}^\rho.
	\end{equation*}
	
	\noindent\textbf{Step 3: Continuity in time.}
	It must be ensured that $\mathcal{N}(u) \in C((0,\infty); L^p(\mathbb{R}^N))$. Let $t > 0$ and $h > 0$. The difference is written as
	\begin{align*}
		\mathcal{N}(u)(t+h) - \mathcal{N}(u)(t) &= \int_t^{t+h} S(t+h-s)f(u(s)) ds \\
		&+ \int_0^t [S(t+h-s) - S(t-s)]f(u(s)) ds.
	\end{align*}
	For the first term, defining the intermediate Lebesgue index $r = p/\rho$ and applying the $L^r \to L^p$ linear smoothing estimate produces a bound proportional to 
	\begin{equation*}
		\int_t^{t+h} (t+h-s)^{-\gamma_{r,p}} \|f(u(s))\|_{L^r} ds.
	\end{equation*}
	Since $\gamma_{r,p} < 1$ and $u \in X_\infty$, this integral vanishes as $h \to 0$. For the second term, the operator $S(t)$ exhibits local H\"older continuity in time, as established in Remark \ref{rem:holder_continuity}. Specifically, the bound 
	\begin{equation*}
		\|S(t+h-s) - S(t-s)\|_{L^r \to L^p} \le C h^\theta (t-s)^{-\tilde{\gamma}}
	\end{equation*}
	guarantees that the integrand is dominated by an $L^1(0,t)$ function. The Lebesgue Dominated Convergence Theorem (cf. Folland \cite{Folland1999}) then ensures the integral vanishes. The case $h < 0$ is analogous, establishing strong continuity.
	
	\noindent\textbf{Step 4: Contraction and global uniqueness in $X_\infty$.}
	Combining the linear and nonlinear estimates yields the bound for the operator $\mathcal{T}$ acting on $u \in X_\infty$:
	\begin{equation*}
		\|\mathcal{T}(u)\|_{X_\infty} \le \delta_0 + C_N \|u\|_{X_\infty}^\rho.
	\end{equation*}
		To establish the strict contraction property, let $u, v \in B_R \subset X_\infty$. Utilizing the algebraic inequality $|f(u) - f(v)| \le c (|u|^{\rho-1} + |v|^{\rho-1}) |u - v|$ and evaluating the $L^r$-norm for $r = p/\rho$, we obtain
	\begin{equation*}
		\|f(u(s)) - f(v(s))\|_{L^r} \le C \left( \|u(s)\|_{L^p}^{\rho-1} + \|v(s)\|_{L^p}^{\rho-1} \right) \|u(s) - v(s)\|_{L^p}.
	\end{equation*}
	Since $u, v \in B_R$, their norms are bounded by $\|u(s)\|_{L^p} \le R \Theta(s)^{-1}$ and $\|v(s)\|_{L^p} \le R \Theta(s)^{-1}$. Consequently, the nonlinear difference satisfies
	\begin{equation*}
		\|f(u(s)) - f(v(s))\|_{L^r} \le 2C R^{\rho-1} \Theta(s)^{-\rho} \|u - v\|_{X_\infty}.
	\end{equation*}
	Given that this bound exhibits the exact same temporal weight profile $\Theta(s)^{-\rho}$ as the single variable estimate $\|f(u(s))\|_{L^r} \le \Theta(s)^{-\rho} \|u\|_{X_\infty}^\rho$ derived in Step 2, applying the identical fractional integration arguments over the intervals $(0, t/2)$ and $(t/2, t)$ directly yields
	\begin{equation*}
		\|\mathcal{T}(u) - \mathcal{T}(v)\|_{X_\infty} \le \widetilde{C}_N R^{\rho-1} \|u - v\|_{X_\infty},
	\end{equation*}
	where $\widetilde{C}_N > 0$ is a structural constant.
	
	The operator $\mathcal{T}$ forms a strict contraction on $B_R$ provided that
	\begin{equation*}
		\delta_0 + C_N R^\rho \le R, \quad \text{and} \quad \widetilde{C}_N R^{\rho-1} \le \frac{1}{2}.
	\end{equation*}
	Setting $R = 2\delta_0$, the strict contraction conditions are simultaneously satisfied if the initial data size $\delta_0$ is chosen sufficiently small such that
	\begin{equation*}
		\delta_0 \le \min \left\{ (2^\rho C_N)^{-\frac{1}{\rho-1}}, \frac{1}{2}(2\widetilde{C}_N)^{-\frac{1}{\rho-1}} \right\}.
	\end{equation*}
	Under this explicit threshold, the Banach Fixed Point Theorem guarantees the existence of a unique global mild solution within the closed ball $B_R$.
	
	To extend this uniqueness to the entire space $X_\infty$, let $u, v \in X_\infty$ be any two mild solutions arising from $u_0$. Restricting the flow to the bounded interval $(0, 2]$ places both $u$ and $v$ within the local resolution space $E_2$. By the local uniqueness established in Theorem \ref{thm:local_existence}, it is rigorously deduced that $u \equiv v$ on $(0, 2]$. 
	
	Let $T_* = \sup\{ \tau > 0 : u \equiv v \text{ on } (0, \tau]\}$. By definition, $T_* \ge 2$. Assume, by contradiction, that $T_* < \infty$. For $t > T_*$, their difference satisfies
	\begin{equation*}
		\mathcal{N}(u)(t) - \mathcal{N}(v)(t) = \int_{T_*}^t S(t-s)[f(u(s)) - f(v(s))] ds.
	\end{equation*}
	Considering that $T_* \ge 2$, the temporal singularity at the origin is bypassed, and both $u$ and $v$ are uniformly bounded in $L^p(\mathbb{R}^N)$ on $[T_*, \infty)$. Let 
	\begin{equation*}
		L = \sup_{s \ge T_*} \left( \|u(s)\|_{L^p}^{\rho-1} + \|v(s)\|_{L^p}^{\rho-1} \right) < \infty.
	\end{equation*}
	Setting the intermediate Lebesgue index $r = p/\rho$ and applying the $L^r \to L^p$ linear smoothing estimate yields
	\begin{equation*}
		\|u(t) - v(t)\|_{L^p} \le C L \int_{T_*}^t (t-s)^{-\gamma_{r,p}} \|u(s) - v(s)\|_{L^p} ds.
	\end{equation*}
	Since $\gamma_{r,p} = q_c/p < 1$, the convolution kernel remains weakly singular. Invoking the generalized Gr\"onwall inequality (cf. Henry \cite{Henry1981}), it is deduced that $\|u(t) - v(t)\|_{L^p} = 0$ on an extended interval $[T_*, T_* + \epsilon]$ for some $\epsilon > 0$, contradicting the maximality of $T_*$. Therefore, $T_* = \infty$, confirming global uniqueness within the resolution space $X_\infty$ and concluding the proof.
\end{proof}

\begin{remark}[Algebraic coupling and the spectral bridge]\label{rem:spectral_bridge}
	Detailing the interplay between the short-time and long-time temporal weights provides a rigorous algebraic foundation for the global-in-time closure of the Duhamel mapping established in Theorem \ref{thm:global_existence}. The partition of the nonlinear Volterra integral at the midpoint $t/2$ for $t > 2$ separates the hereditary memory effects into two distinct spectral regimes, denoted by
	\begin{equation*}
		I_1(t) = \int_0^{t/2} S(t-s) f(u(s)) \, ds, \quad \text{and} \quad I_2(t) = \int_{t/2}^t S(t-s) f(u(s)) \, ds.
	\end{equation*}
	Evaluating the structural integrability of the transient branch $I_1(t)$, the elapsed time variable satisfies $t-s \ge t/2 > 1$. Consequently, Lemma \ref{lem:smoothing_asymptotic} supplies the unperturbed asymptotic $L^m \to L^p$ regularizing effect. Recalling that $\|f(u(s))\|_{L^m} = \|u(s)\|_{L^p}^\rho$, the continuous embedding associated with the resolution norm $\|u\|_{X_\infty}$ yields
	\begin{equation}\label{eq:i1_integral_split}
		\|I_1(t)\|_{L^p} \le C \int_0^{t/2} (t-s)^{-\beta_{m,p}} \Theta(s)^{-\rho} \, ds \cdot \|u\|_{X_\infty}^\rho \le C t^{-\beta_{m,p}} \int_0^{t/2} \Theta(s)^{-\rho} \, ds \cdot \|u\|_{X_\infty}^\rho,
	\end{equation}
	where the piecewise temporal weight satisfies $\Theta(s)^{-\rho} = \max\left( s^{-\rho \gamma_{q_c,p}}, s^{-\rho \beta_{m,p}} \right)$. For the integral in \eqref{eq:i1_integral_split} to converge uniformly as $t \to \infty$, the system must simultaneously satisfy a local integrability condition near the origin and a global decay condition at infinity.
	
	The singularity at the origin is governed by the high-frequency smoothing index $\gamma_{q_c,p}$. Invoking the explicit formulation of the critical Lebesgue threshold $q_c := \frac{N(\rho-1)(1+\alpha_\infty)}{2}$, it is observed that
	\begin{equation*}
		\gamma_{q_c,p} = \frac{N(1+\alpha_\infty)}{2}\left(\frac{1}{q_c} - \frac{1}{p}\right) = \frac{q_c}{\rho-1}\left(\frac{1}{q_c} - \frac{1}{p}\right) = \frac{1}{\rho-1}\left(1 - \frac{q_c}{p}\right).
	\end{equation*}
	Enforcing local integrability requires $\rho \gamma_{q_c,p} < 1$. Substituting the rewritten expression for $\gamma_{q_c,p}$, this inequality transforms into
	\begin{equation}\label{eq:local_cond_p}
		\frac{\rho}{\rho-1}\left(1 - \frac{q_c}{p}\right) < 1 \implies \rho - \frac{\rho q_c}{p} < \rho - 1 \implies p < \rho q_c.
	\end{equation}
	
	Concurrently, the integrability of the long-time memory tail requires $\rho \beta_{m,p} > 1$. Setting $p = \rho m$, the low-frequency regularizing index simplifies to
	\begin{equation*}
		\beta_{m,p} = \frac{N(1+\alpha_0)}{2}\left(\frac{1}{m} - \frac{1}{p}\right) = \frac{N(1+\alpha_0)}{2}\left(\frac{\rho-1}{p}\right) = \frac{m_c}{p},
	\end{equation*}
	where $m_c := \frac{N(\rho-1)(1+\alpha_0)}{2}$ represents the non-local critical mass threshold. Thus, the asymptotic non-explosion condition $\rho \beta_{m,p} > 1$ becomes
	\begin{equation}\label{eq:global_cond_p}
		\frac{\rho m_c}{p} > 1 \implies p < \rho m_c \implies m < m_c.
	\end{equation}
	
	The intersection of these integrability conditions reveals a fundamental topological constraint on the choice of the base Lebesgue space. From the high-frequency smoothing requirement $\gamma_{r,p} = q_c/p < 1$ established in the local theory (Section 4), we obtain the strict lower bound $p > q_c$. Simultaneously, the system enforces two distinct upper bounds: the local integrability limit $p < \rho q_c$ derived in \eqref{eq:local_cond_p}, and the global non-explosion condition $p < \rho m_c$ from \eqref{eq:global_cond_p}. 
	
	Given that the multi-scale kernel dictates the structural crossover $\alpha_0 < \alpha_\infty$, the non-local critical mass is unconditionally smaller than the local critical threshold ($m_c < q_c$). Consequently, the global upper bound is tighter ($\rho m_c < \rho q_c$), rendering the local condition \eqref{eq:local_cond_p} automatically satisfied for any admissible index. The topological window collapses securely to $q_c < p < \rho m_c$. 
	Since the relationship between the Lebesgue indices dictates $p = \rho m$, dividing this viable window by the non-linearity power $\rho$ yields the exact structural requirement for the intersection parameter
	\begin{equation*}
		m \in \left( \frac{q_c}{\rho}, m_c \right).
	\end{equation*}
	For this interval to be non-empty, the endpoints must satisfy the geometric compatibility relation
	\begin{equation*}
		\frac{q_c}{\rho} < m_c \implies \frac{N(\rho-1)(1+\alpha_\infty)}{2\rho} < \frac{N(\rho-1)(1+\alpha_0)}{2}.
	\end{equation*}
	Since $\rho > 1$, the common structural prefactor $\frac{N(\rho-1)}{2}$ is positive and can be canceled. Isolating the nonlinearity power $\rho$ yields
	\begin{equation*}
		\frac{1+\alpha_\infty}{\rho} < 1+\alpha_0 \implies \rho > \frac{1+\alpha_\infty}{1+\alpha_0}.
	\end{equation*}
	
	This algebraic relation demonstrates that the spectral bridge constraint $\rho > \frac{1+\alpha_\infty}{1+\alpha_0}$ is not an artifact of the time-weighted spaces, but a strict structural necessity. It represents the exact threshold below which the power-type nonlinear amplification overpowers the dual-scale anomalous dissipation, ultimately causing the contraction mapping to collapse.
\end{remark}

\begin{remark}\label{rem:thresholds}
	Two topological features of the resolution space $X_\infty$ are highlighted.
	
	\noindent\textbf{(i) Continuity at the origin and the integrability gap:} Although the mild solution satisfies $u \in C((0,\infty); L^p(\mathbb{R}^N))$, strong continuity at $t=0$ structurally fails in the $L^p(\mathbb{R}^N)$ topology. Considering that the global framework requires $m > q_c/\rho$, the target resolution space index satisfies $p = \rho m > q_c$. Consequently, the space $L^p(\mathbb{R}^N)$ lies strictly outside the natural interpolation range of the initial data $u_0 \in L^{q_c}(\mathbb{R}^N) \cap L^m(\mathbb{R}^N)$. Due to the instantaneous regularizing effect of the resolvent, this integrability gap implies that $\limsup_{t \to 0^+} \|u(t)\|_{L^p} = \infty$ whenever $u_0 \notin L^p(\mathbb{R}^N)$, thereby necessitating the singular temporal weight $\Theta(t)$. However, the initial datum is continuously attained in the topology of the critical base spaces, satisfying
	\begin{equation*}
		\lim_{t \to 0^+} \|u(t) - u_0\|_{L^{q_c}(\mathbb{R}^N) \cap L^m(\mathbb{R}^N)} = 0.
	\end{equation*}
	This follows from the strong continuity of the resolvent operator $S(t)$ in Lebesgue spaces without initial singularity and standard density arguments analogous to the local theory.
	
	\noindent\textbf{(ii) The non-existence threshold $\rho \le \rho_0$:} The condition $\rho > \rho_0$ is structurally optimal for the closure of this multi-scale framework. If $\rho \le \rho_0$, global existence fails due to two distinct mathematical obstructions. First, if $\rho \le \rho_F$, the long-time non-local interaction integrated over the low-frequency tail, namely
	\begin{equation*}
		\int_1^\infty s^{-\rho\beta_{m,p}} \, ds,
	\end{equation*}
	diverges algebraically or logarithmically. This divergence precludes the asymptotic integrability required to close the Volterra contraction mapping. In alignment with the classical Fujita phenomenon, the collapse of this upper bound indicates that the non-local anomalous dissipation is entirely counterbalanced by the instantaneous nonlinear reaction; thus, non-trivial mild solutions satisfying suitable initial mass or sign configurations are expected to undergo finite-time singularity formation. Since tracking the exact weak space-time blow-up thresholds for this generalized dual-scale class of admissible kernels requires a disparate framework based on specialized nonlinear test functions and variational energy methods---extending the explicit threshold methodologies recently developed by the author in \cite{deA26} for the standard fractional diffusion-wave equation---this non-existence inquiry remains an open problem to be addressed in a forthcoming independent work. Second, if $\rho \le \frac{1+\alpha_\infty}{1+\alpha_0}$, the admissible Lebesgue intersection interval $\left(\frac{q_c}{\rho}, m_c\right)$ collapses and becomes empty. This represents a fundamental algebraic barrier where the nonlinear amplification power lacks the capacity to bridge the short-time super-diffusive regularization near the origin with the long-time relaxation at infinity.
\end{remark}

	Concluding the analysis of the global dynamics, the asymptotic profile of the nonlinear solution is investigated. In classical nonlinear parabolic and dispersive theories, critical nonlinearities often alter the fundamental decay rate of the system, leading to anomalous steepening or distinct self-similar profiles. However, the constraint $\rho > \rho_F$ limits the temporal mass of the nonlinear interactions. Considering that the nonlinear source is integrable in the low-frequency regime, it acts as a transient perturbation. Consequently, the long-time behavior of the nonlinear flow is governed by the linear fractional profile dictated by $\alpha_0$. This stability is formalized in the following corollary.
	
	\begin{co}\label{cor:asymptotic_decay}
		Let $u \in X_\infty$ be the global mild solution obtained in Theorem \ref{thm:global_existence}. Then, the nonlinear source satisfies $f(u) \in L^1(0, \infty; L^m(\mathbb{R}^N))$. Furthermore, $u$ exhibits the asymptotic linear decay
		\begin{equation*}
			\limsup_{t \to \infty} t^{\beta_{m,p}} \|u(t)\|_{L^p(\mathbb{R}^N)} \le C_L \|u_0\|_{L^m} + C_\infty \|u\|_{X_\infty}^\rho,
		\end{equation*}
		where $C_\infty > 0$ is a constant depending exclusively on the total nonlinear mass.
	\end{co}
	
	\begin{proof}
		Since $u \in X_\infty$ and $p = \rho m$, it follows that $\|f(u(s))\|_{L^m} = \|u(s)\|_{L^p}^\rho$. According to the definition of the resolution space, $\|u(s)\|_{L^p} \le \Theta(s)^{-1} \|u\|_{X_\infty}$. As $\Theta(s)^{-1} = \min(s^{-\gamma_{q_c, p}}, s^{-\beta_{m, p}})$, the nonlinear source satisfies the global bound
		\begin{equation*}
			\|f(u(s))\|_{L^m} \le \min \left( s^{-\rho \gamma_{q_c, p}}, s^{-\rho \beta_{m, p}} \right) \|u\|_{X_\infty}^\rho, \quad \forall s > 0.
		\end{equation*}
		Since $\rho \gamma_{q_c, p} < 1$, the singularity near the origin is locally integrable; concurrently, the condition $\rho > \rho_F$ ensures $\rho \beta_{m, p} > 1$, guaranteeing integrability at infinity. Integrating over $(0, \infty)$ yields
		\begin{align*}
			\int_0^\infty \|f(u(s))\|_{L^m} ds &\le \left( \int_0^1 s^{-\rho \gamma_{q_c, p}} ds + \int_1^\infty s^{-\rho \beta_{m, p}} ds \right) \|u\|_{X_\infty}^\rho \\
			&:= \widetilde{C}_\infty \|u\|_{X_\infty}^\rho < \infty,
		\end{align*}
		which establishes that $f(u) \in L^1(0, \infty; L^m(\mathbb{R}^N))$.
		
		Evaluating the asymptotic limit requires revisiting the Duhamel splitting: 
		\begin{equation*}
			u(t) = S(t)u_0 + I_1(t) + I_2(t).
		\end{equation*}
		Multiplying by the temporal weight $t^{\beta_{m,p}}$ and applying the triangle inequality yields
		\begin{equation*}
			t^{\beta_{m,p}} \|u(t)\|_{L^p} \le t^{\beta_{m,p}} \|S(t)u_0\|_{L^p} + t^{\beta_{m,p}} \|I_1(t)\|_{L^p} + t^{\beta_{m,p}} \|I_2(t)\|_{L^p}.
		\end{equation*}
		Lemma \ref{lem:smoothing_asymptotic} guarantees $t^{\beta_{m,p}} \|S(t)u_0\|_{L^p} \le C_L \|u_0\|_{L^m}$. Furthermore, Theorem \ref{thm:global_existence} establishes that $\|I_2(t)\|_{L^p}$ decays faster than $t^{-\beta_{m,p}}$, implying $\lim_{t \to \infty} t^{\beta_{m,p}} \|I_2(t)\|_{L^p} = 0$.
		
		For the transient memory component $I_1(t)$, applying the $L^m \to L^p$ linear bound yields
		\begin{equation*}
			t^{\beta_{m,p}} \|I_1(t)\|_{L^p} \le C \int_0^{t/2} \left( \frac{t}{t-s} \right)^{\beta_{m,p}} \|f(u(s))\|_{L^m} ds.
		\end{equation*}
		As the integration domain is restricted to $s \in (0, t/2)$, the temporal ratio satisfies $\frac{t}{t-s} \le 2$. The corresponding integrand is therefore dominated by the $L^1(0, \infty)$ function $2^{\beta_{m,p}} \|f(u(s))\|_{L^m}$. Since $\lim_{t \to \infty} \left( \frac{t}{t-s} \right)^{\beta_{m,p}} = 1$ pointwise for almost all $s > 0$, invoking the Lebesgue Dominated Convergence Theorem provides
		\begin{equation*}
			\limsup_{t \to \infty} t^{\beta_{m,p}} \|I_1(t)\|_{L^p} \le C \int_0^\infty \|f(u(s))\|_{L^m} ds \le C \widetilde{C}_\infty \|u\|_{X_\infty}^\rho.
		\end{equation*}
		Defining $C_\infty = C \widetilde{C}_\infty$ completes the proof.
	\end{proof}

	\appendix
	
\section{Spatial localization and Peetre's inequality}
Securing the uniform lower bound $\|F_\lambda\|_{L^q(\mathbb{R}^N)} \ge c_1 > 0$ required for the proof of instantaneous norm inflation in Theorem \ref{thm:ill_posedness} requires establishing a pointwise polynomial decay for the pre-limit profile $F_\lambda(y)$ that is independent of the scaling parameter $\lambda \ge 1$. Consider the localized convolution operator defining the uniform envelope function $\mathcal{E}(y)$, denoted for each fixed $\theta \in (0, \tau]$ by
\begin{align}\label{eq:app_conv_def}
	J_\theta(y) &:= \left[ \theta^{-N/\sigma} \mathcal{G}_t(\cdot \theta^{-1/\sigma}) \ast f(K(\theta, \lambda)\phi) \right](y)\\ \nonumber
	& = \int_{\mathbb{R}^N} \theta^{-N/\sigma} \mathcal{G}_t((y-z)\theta^{-1/\sigma}) f(K(\theta, \lambda)\phi(z)) \, dz.
\end{align}
A direct pointwise analysis of \eqref{eq:app_conv_def} reveals a temporal prefactor $\theta^{-N/\sigma}$ that diverges as $\theta \to 0^+$ for any dimension $N \ge 1$ (since $\sigma < 2$). To resolve this singularity, a coordinate transformation is applied to the spatial integration. Setting $z = y - w \theta^{1/\sigma}$, the differential volume element transforms as $dz = \theta^{N/\sigma} dw$. Substituting these relations into \eqref{eq:app_conv_def} yields the regularized integral
\begin{equation}\label{eq:regularized_j_theta}
	J_\theta(y) = \int_{\mathbb{R}^N} \mathcal{G}_t(w) f(K(\theta, \lambda)\phi(y - w \theta^{1/\sigma})) \, dw,
\end{equation}
where the temporal prefactor is absorbed by the Jacobian of the dilation.

The localization weight $(1+|y|)^{N+2}$ is introduced to control the spatial decay away from the origin. Invoking Peetre's inequality (cf. Bergh and L\"ofstr\"om \cite{Bergh1976}) for the triplet $\left(y, \, y - w \theta^{1/\sigma}, \, w \theta^{1/\sigma}\right)$ with exponent $s = N+2$, the weight is bounded by
\begin{equation}\label{eq:peetre_applied_app}
	(1+|y|)^{N+2} \le 2^{N+2} \left(1 + |y - w \theta^{1/\sigma}|\right)^{N+2} \left(1 + |w \theta^{1/\sigma}|\right)^{N+2}.
\end{equation}
Considering that the evaluation is restricted to $\theta \in (0, \tau^\ast]$, the geometric scaling parameter satisfies $\theta^{1/\sigma} \le (\tau^\ast)^{1/\sigma}$. Introducing the structural constant $C_{\tau^\ast} = \max\left(1, (\tau^\ast)^{\frac{N+2}{\sigma}}\right)$ yields the uniform algebraic bound $\left(1 + |w \theta^{1/\sigma}|\right)^{N+2} \le C_{\tau^\ast} (1+|w|)^{N+2}$. Substituting this estimate into \eqref{eq:peetre_applied_app} and weighting the regularized convolution \eqref{eq:regularized_j_theta} produces
\begin{equation}\label{eq:weighted_j_integral}
	\begin{split}
		(1+|y|)^{N+2} |J_\theta(y)| &\le 2^{N+2} C_{\tau^\ast} \int_{\mathbb{R}^N} (1+|w|)^{N+2} |\mathcal{G}_t(w)| \\
		&\quad \times \left(1 + |y - w \theta^{1/\sigma}|\right)^{N+2} \left| f(K(\theta, \lambda)\phi(y - w \theta^{1/\sigma})) \right| dw.
	\end{split}
\end{equation}

The uniform bounds can now be established. By the explicit characterization of the localized singularity detailed in Remark \ref{rem:kernel_singularity}, the spatial kernel $\mathcal{G}_t(w)$ associated with the H\"ormander class $S^{-2}_{1,0}$ is locally integrable near the origin for any dimension $N \ge 1$, because the radial measure $r^{N-1}$ neutralizes both the algebraic pole $r^{-(N-2)}$ and the logarithmic singularity. Combining this local integrability with rapid polynomial decay at infinity ensures the global integrability of the weighted kernel; specifically,
\begin{equation}\label{eq:kernel_l1_weighted_final}
	\int_{\mathbb{R}^N} (1+|w|)^{N+2} |\mathcal{G}_t(w)| \, dw := \widetilde{C}_{\mathcal{G}} < \infty,
\end{equation}
where the integral constant $\widetilde{C}_{\mathcal{G}} > 0$ is independent of the transient time configurations. Furthermore, considering that the target profile belongs to the Schwartz class $\phi \in \mathcal{S}(\mathbb{R}^N)$ and the linear operator $K(\theta, \lambda)$ acts boundedly on $\mathcal{S}(\mathbb{R}^N)$, the power-type nonlinearity preserves this rapid polynomial decay. This yields the uniform envelope bound
\begin{equation}\label{eq:schwartz_nonlin_bound_final}
	\sup_{x \in \mathbb{R}^N} (1+|x|)^{N+2} |f(K(\theta, \lambda)\phi(x))| \le C_{\mathcal{S}} < \infty.
\end{equation}
Applying the supremum bound $C_{\mathcal{S}}$ to the translated nonlinear source in \eqref{eq:weighted_j_integral}, and resolving the spatial integration over the weighted kernel via \eqref{eq:kernel_l1_weighted_final}, isolates the scaling parameters, yielding
\begin{equation*}
	(1+|y|)^{N+2} |J_\theta(y)| \le 2^{N+2} C_{\tau^\ast} C_{\mathcal{S}} \int_{\mathbb{R}^N} (1+|w|)^{N+2} |\mathcal{G}_t(w)| \, dw = 2^{N+2} C_{\tau^\ast} C_{\mathcal{S}} \widetilde{C}_{\mathcal{G}}.
\end{equation*}
Taking the supremum over $\theta \in (0, \tau^\ast]$ on both sides establishes the uniform algebraic localization
\begin{equation*}
	|F_\lambda(y)| \le \int_0^{\tau^\ast} |J_\theta(y)| \, d\theta \le C_0 (1+|y|)^{-(N+2)},
\end{equation*}
where the constant $C_0 := 2^{N+2} \tau^\ast C_{\tau^\ast} C_{\mathcal{S}} \widetilde{C}_{\mathcal{G}} > 0$ is independent of the asymptotic modulation parameter $\lambda \ge 1$. This completes the proof of spatial localization, confirming that the spatial mass of the first Picard iterate remains confined to a compact region, which ultimately validates the sharpness of the supercritical ill-posedness threshold.

	\
	
	\noindent{\bf \large Acknowledgement and declarations}
	
	\
	
	\noindent Bruno de Andrade is partially supported by CNPQ (grant 310384/2022-2) and FAPITEC/SE (grant 019203.01303/2024-1).\\
	
	\noindent{\bf Declaration of Generative AI and AI-assisted technologies in the writing process:}	During the preparation of this work, the author used Gemini to improve the English language. After using this tool, the author reviewed and edited the content as needed and take full responsibility for the final version of the manuscript.

\end{document}